\documentclass{amsart}

\usepackage{amsmath,amssymb}
\usepackage{enumerate}
\usepackage{xcolor}
\usepackage{graphicx}
\usepackage{hyperref}

\newtheorem{theorem}{Theorem}[section]
\newtheorem{corollary}[theorem]{Corollary}
\newtheorem{lemma}{Lemma}[section]
\newtheorem{proposition}{Proposition}[section]
\theoremstyle{remark}
\newtheorem{remark}{Remark}[section]

\theoremstyle{definition}
\newtheorem{definition}{Definition}[section]

\def\N{{\mathbb{N}}}
\def\Z{{\mathbb{Z}}}

\def\MN{{\mathbb{N}}}

\def\MA{{\mathbb{A}}}
\def\MB{{\mathbb{B}}}

\makeatletter
\newsavebox{\@brx}
\newcommand{\llangle}[1][]{\savebox{\@brx}{\(\m@th{#1\langle}\)}%
  \mathopen{\copy\@brx\kern-0.5\wd\@brx\usebox{\@brx}}}
\newcommand{\rrangle}[1][]{\savebox{\@brx}{\(\m@th{#1\rangle}\)}%
  \mathclose{\copy\@brx\kern-0.5\wd\@brx\usebox{\@brx}}}
\makeatother

\newcommand{\nstdN}{\widetilde{\mathbb N}}
\newcommand{\nstdM}{\widetilde{\mathbb N}}
\newcommand{\ndots}{\ldots}
\newcommand{\List}{\mathrm{List}}

\newcommand{\Th}{{\mathop{\mathrm{Th}}}}

\newcommand{\an}[1]{{\color{red}AN: #1}}

\begin{document}

\title[Nonstandard free groups]{Nonstandard free groups}

\author{Alexei Miasnikov}
\email{amiasnik@stevens.edu}
\author{Andrey Nikolaev}
\email{anikolae@stevens.edu}


\begin{abstract}
Interpretation of a structure $\mathbb A$ in $\mathbb B$ allows to produce structures elementarily equivalent to $\mathbb A$ given those elementarily equivalent to $\mathbb B$. In particular, interpretation of the free group in $\mathbb N$ enables us to introduce and study a family of elementary free groups, which we call nonstandard free groups. More generally, for a wide class of groups we introduce nonstandard models arising from interpretation in $\mathbb N$. We exploit interpretation to show that under mild assumptions, ultrapowers of a group can be viewed as nonstandard models of that group. This leads us to describe the structure of the ultrapowers in terms of structure of nonstandard models of natural numbers, offering insight into a longstanding question of Malcev. We also introduce fundamentals of nonstandard combinatorial group theory such as the notions of nonstandard subgroups, nonstandard normal subgroups, and nonstandard group presentations.
\end{abstract}


\maketitle
\tableofcontents

\section{Motivation and main ideas}\label{se:motivation}

Let $G$ be a countable group where the multiplication is not absolutely bizarre, say, there is a generating set for which the word problem is decidable, or recursively enumerable, or just arithmetic, i.e., it is defined by a formula of arithmetic (definable in arithmetic). Then $G$ is interpretable in $\Z$ via some interpretation $\Gamma$, symbolically $G \cong \Gamma(\Z)$ (see Section~\ref{se:interpretation} for relevant definitions). If $ \widetilde \Z \equiv \Z$ then $\Gamma$ interprets a group $\widetilde G = \Gamma(\widetilde \Z)$ in $\widetilde \Z$, which we call a  \emph{nonstandard model} of $G$. 
If $G$ is finitely generated, then the group $\Gamma(\widetilde \Z)$ does not depend on the particular interpretation $\Gamma$, so if $G \cong \Gamma(\Z)$ and $G \cong \Delta(\Z)$ then $\Gamma(\widetilde \Z) \cong \Delta(\widetilde \Z)$ for any ring $\widetilde \Z \equiv \Z$ (Theorem~\ref{th:equiv2}).
In this case, the nonstandard versions of $G$ are completely determined by $G$ and $\widetilde \Z$, so this is ultimately about groups, not interpretations. With this in mind, we denote $\Gamma(\widetilde{\mathbb Z})$ by $G(\widetilde \Z)$ or $\widetilde{G}$. If $G$ is not finitely generated, we still denote $\Gamma(\widetilde \Z)$ by $G(\widetilde \Z)$ or $\widetilde{G}$ if the interpretation $\Gamma$ is irrelevant or clear from the context.



The central idea is  that the nonstandard groups $G(\widetilde \Z)$ are the closest versions of $G$ which, on the one hand, preserve the principal properties of $G$ that we want to investigate, but on the other hand, are general enough so many non-essential features of $G$ disappear in them.
Furthermore, since $G(\widetilde \Z) = \Gamma(\widetilde \Z)$ one can see and recover from $\Gamma$ and $\widetilde \Z$ the general constructions and frames on which these essential properties depend.
This leads to a better understanding of $G$ through the variety of its nonstandard models and gives a general method of approaching  various open problems on $G$.
Some of these problems can be solved using nonstandard groups directly, while for others, nonstandard groups show where to look for solutions and what kind of tools to use. In the present paper, we restrict our attention to the former, as we explain below in Sections~\ref{se:FOC} and~\ref{se:ultrapowers}.

When $G$ is a free group $G=F$, we call the corresponding nonstandard model $F(\widetilde{\mathbb Z})$ the nonstandard free group, and denote it by $\widetilde{F}$. In many regards, it behaves like the standard free group $F$, with the key difference that $\widetilde{F}$ ``natively'' admits certain infinite products. For that, the principal instrument is so-called nonstandard list superstructure (see Section~\ref{se:nonstd_list}). This feature propagates to numerous other usual notions of group theory, leading us to introduce nonstandard subgroups, nonstandard group presentations, and the first isomorphism theorem, which are explored in Sections~\ref{se:basic_structure} and~\ref{se:presentations}.

\subsection{First-order classification and elementary equivalence} 
\label{se:FOC}

Denote by $Th(G)$ the first-order theory of $G$ in the language of group theory $L = \{\cdot, ^{-1},1\}$.
One of the initial motivations of this project is in developing a general approach for  characterizing algebraically all the groups $H$ that are elementarily equivalent to a given group $G$. This is the well-known Tarski first-order classification problem for groups. 

It follows from the properties of interpretations that $G(\widetilde \Z)\equiv G$ (Theorem~\ref{th:equiv0}) for any $\widetilde \Z \equiv \Z$.
However, in general, not all models of $Th(G)$ are nonstandard versions of $G$.
For example, this is the case for a free non-abelian group $F$ or a non-abelian torsion-free hyperbolic group $G$ (it comes from the solution of the Tarski problem for these groups \cite{Kharlampovich-Myasnikov:2006,Sela:2006}).
Nevertheless, the nonstandard groups $G(\widetilde{\Z})$ are important models of $Th(G)$; they are ``more equivalent'' to $G$ than all other models of $Th(G)$. In fact, they are equivalent to $G$ in a much stronger logic than the first-order logic in the language of groups. This is a new type of logic, which deserves a study in its own right. We offer a brief introductory glimpse into this in Section~\ref{se:summation_arb_superstructre}, in particular, Subsection~\ref{se:wso_translation}.

Furthermore, it turns out that the nonstandard models of $G$ give an instrument to understand all other models of $Th(G)$.
Indeed, first, like in the nonstandard arithmetic, the group $G$ embeds into each $G(\widetilde \Z)$ as the \emph{standard part} of $G(\widetilde \Z)$ and this embedding is elementary (Theorem~\ref{th:elementary_submodel}).
Second, every model $H$ of $Th(G)$ is an elementary submodel of some nonstandard model $G(\widetilde \Z)$ of the same cardinality (Theorem~\ref{th:equivalent_models}). Third, every saturated model of $Th(G)$ (under the generalized continuum hypothesis) is a nonstandard version $G(\widetilde \Z)$ for some saturated model $\widetilde \Z$ of arithmetic (Theorem~\ref{th:saturated}).

Observe that a nonstandard free group $\widetilde{F}=F(\widetilde{\mathbb Z})$ is not finitely generated if $\widetilde{\mathbb Z}\neq \mathbb Z$. Indeed, a finitely generated group elementarily equivalent to a free group must be hyperbolic, and therefore has cyclic centralizers. However, it is easy to see that in $\widetilde{F}$, centralizers are isomorphic to the additive group $\widetilde{\mathbb Z}^+$ (Proposition~\ref{pr:centralizers}), which is not cyclic (see Section~\ref{se:nonstandard_arithmetic}).

Motivated by questions posed in~\cite{Kharlampovich-Myasnikov-Sklinos:2020} regarding countable groups elementarily equivalent to a free group $F$, the paper~\cite{Kharlampovich-Natoli:2024} describes groups $G\equiv F$ with cyclic centralizers, and also gives the example $\mathbb Z*(\mathbb Z \oplus \mathbb Q^+)$ of a group elementarily equivalent to $F$ but with non-cyclic centralizers. All nonstandard free groups $\widetilde{F}$ are new examples of groups with non-cyclic centralizers. In fact, since all centralizers in $\widetilde{F}$ are non-cyclic, it follows that the above group $\mathbb Z*(\mathbb Z + \mathbb Q^+)$ is not a nonstandard free group.

\subsection{Ultrapowers and nonstandard groups}  \label{se:ultrapowers} Another piece of motivation is to address an old problem due to Malcev: for a free nonabelian  group $F$, a set  $I$, and a non-principal ultrafilter $D$ on $I$ describe the algebraic structure of the ultrapower $F^I/D$.
At the time, the problem was ultimately related to the Tarski problem of whether or not two non-abelian free groups, say $F_2$ and $F_3$, are elementarily equivalent to each other, via the Keisler--Shelah theorem (two groups are elementarily equivalent iff their non-trivial ultrapowers are isomorphic).
Later, this question was asked for other groups $G$ as well.
If $G$ is ``finitely dimensional'', say algebraic over a field $K$ (or over a suitable ring $R$), then $G^I/D$ is again algebraic with the same algebraic scheme, but over $K^I/D$  (or $R^I/D$), so we know the algebraic structure of $G^I/D$.
This idea was successfully applied to the  classical matrix groups or nilpotent groups, but it does not work for ``infinitely dimensional'' groups like $F$, or nonabelian torsion-free hyperbolic groups.
However, if $G \simeq \Gamma(\Z)$ for some $\Gamma$, then $G^I/D \cong G(\widetilde \Z)$, where $\widetilde \Z \cong \Z^I/D \equiv \Z$; here $\Gamma$ plays a part of ``first-order algebraic scheme'' for $G$.
Hence $G^I/D$ is a nonstandard model of $G$ with some extra properties that come from the ring $\Z^I/D$.
This gives a method to approach questions on the algebraic structure of $G^I/D$ viewed as $G(\widetilde \Z)$.
In particular, returning to the original Malcev's question, we can see that $F^I/D \cong F(\widetilde \Z)$ is a nonstandard free group (see Section~\ref{se:free_group}) with basis $X$, the same object we introduced earlier in this section. This allows to describe $F^I/D$ in the same terms as nonstandard models in general, see Theorems~\ref{th:free_group_ultrapower}, \ref{th:free_group_ultrapower}, \ref{th:arbitrary_group_ultrapower_gen}; in fact, it is easy to see that the same applies to a wide class of algebraic structures, see Theorem~\ref{th:arbitrary_structure_ultrapower_gen}.

\section{Interpretability}\label{se:prelims}
In this section, we discuss various types of interpretability of structures; some of them are well-known, others are new. Each of them serves different purposes. Our focus is mostly on the regular interpretability, which suits well to the classical first-order classification problem. 

Let $L$ be a language (signature) with a set of functional (operational) symbols $\{f, \ldots\}$ together with their arities $n_f \in \N$, a set of constant symbols $\{c, \ldots\}$ and a set of relation (or predicate) symbols $\{R, \ldots \}$ coming together with their arities $n_R \in \N$. We write $f(x_1, \ldots,x_n)$ or $R(x_1,\ldots,x_n)$ to show that $n_f = n$ or $n_R=n$. The standard language of groups $\{\cdot\,,\,^{-1},e\}$ includes the symbol $\cdot$ for binary operation of multiplication, the symbol $^{-1}$ for unary operation of inversion, and the symbol $e$ for the group unit; the standard language of unitary rings is $\{+,\,\cdot\,,0,1\}$. An interpretation of a constant symbol $c$ in a set $A$ is an element $c^A\in A$. For a functional symbol $f$ an interpretation in $A$ is a function $f^A\colon A^{n_f}\to A$, and for a predicate $R$ it is a set $R^A\subseteq A^{n_R}$.

An algebraic structure in the language $L$ (an $L$-structure) with the base set $A$ is denoted by $\MA = \langle A; L\rangle$, or simply by $\MA = \langle A; f, \ldots,R,\ldots,c, \ldots \rangle$, or by $\MA = \langle A; f^A, \ldots,R^A,\ldots,c^A, \ldots \rangle$. For a given structure $\MA$ by $L(\MA)$ we denote the language of $\MA$. 

We usually denote variables by small letters $x,y,z, a,b, u,v, \ldots$, while the same symbols with bars $\bar x, \bar y, \ldots$ denote tuples of the corresponding variables, say $\bar x = (x_1, \ldots,x_n)$, and furthermore, $\bar{\bar x}$ is a tuple of tuples $\bar{\bar x}=(\bar x_1,\ldots,\bar x_m)$, where ${|\bar x_i|=n}$.

By default, we consider the sets $\mathbb N$, $\mathbb Z$, $\mathbb Q$, $\mathbb R$ and $\mathbb C$ in the language $\{+,\cdot,0,1\}$.
Note that definable sets in arithmetic $\mathbb N$, the so-called arithmetic sets, are well-studied.
In particular, it is known that every computably enumerable set in $\mathbb N$ is definable. The same holds for the ring of integers $\mathbb Z$. Furthermore, every element of $\mathbb N$ (or $\mathbb Z$, or $\mathbb Q$) is absolutely definable in 
$\mathbb N$ ($\mathbb Z$, $\mathbb Q$), so every definable set in $\mathbb N$ ($\mathbb Z$, $\mathbb Q$) is absolutely definable.

\subsection{Interpretations}\label{se:interpretation}

\begin{definition} \label{de:interpretation}\label{de:interpretable} 
An algebraic structure $\MA = \langle A;f,\ldots,R,\ldots,c,\ldots\rangle$ is $0$-{\em interpretable} (or {\em absolutely interpretable}, or {\em interpretable without parameters}) in an algebraic structure $\MB=\langle B;L(\MB)\rangle$ if the following conditions hold:
\begin{enumerate}[1)]
\item there is a subset $A^\ast \subseteq B^n$ $0$-definable in $\MB$, 
\item there is an equivalence relation $\sim$ on $A^\ast$ $0$-definable in $\MB$, 
\item there are interpretations $f^A, \ldots$, $R^A,\ldots$, $c^A, \ldots$ of the symbols $f, \ldots,$ $R,\ldots,$ $c, \ldots$ on the quotient set $A^\ast /{\sim}$, all $0$-definable in $\MB$, 
\item the structure $\MA^\ast=\langle A^\ast /{\sim}; f^A, \ldots,R^A,\ldots, c^A, \ldots \rangle$ is $L(\MA)$-isomorphic to~$\MA$.
\end{enumerate}
\end{definition}

The structure $\mathbb A^\ast$ from Definition~\ref{de:interpretation} is completely described by the following set of formulas in the language $L(\mathbb B)$:
\begin{equation} \label{eq:code}
\Gamma = \{U_\Gamma(\bar x), E_\Gamma(\bar x, \bar x^\prime), Q_\Gamma(\bar x_1, \ldots,\bar x_{t_Q}) \mid Q \in L(\mathbb A)\},
\end{equation}
where $\bar x$, $\bar x^\prime$ and $\bar x_i$ are $n$-tuples of variables.
Namely, $U_\Gamma$ defines in $\mathbb B$ a set $A_\Gamma = U_\Gamma(B^n) \subseteq B^n$ (the set $A^\ast$ in Definition~\ref{de:interpretation}), $E_\Gamma$ defines an equivalence relation $\sim_\Gamma$ on $A_\Gamma$ (the equivalence relation $\sim$ in Definition~\ref{de:interpretation}), and the formulas $Q_\Gamma$ define the preimages of the graphs for constants, functions, and predicates $Q\in L(\mathbb A)$ on the quotient set $A_\Gamma/{\sim_\Gamma}$ in such a way that the structure $\Gamma(\mathbb B) = \langle A_\Gamma/{\sim_\Gamma}; L(\mathbb A) \rangle $ is isomorphic to $\mathbb A$.
Here $t_c=1$ for a constant $c\in L(\mathbb A)$, $t_f=n_f+1$ for a function $f\in L(\mathbb A)$ and $t_R=n_R$ for a predicate $R\in L(\mathbb A)$.
Note that we interpret a constant $c \in L(\mathbb A)$ in the structure $\Gamma(\mathbb B)$ by the $\sim_\Gamma$-equivalence class of some tuple $\bar b_c \in A_\Gamma$ defined in $\mathbb B$ by the formula $c_\Gamma(\bar x)$.
The number $n$ is called the {\em dimension} of $\Gamma$, denoted $n = \dim \Gamma$.

We refer to $\Gamma$ as the {\em interpretation code} or just the {\em code} of the interpretation $\mathbb A\rightsquigarrow\mathbb B$.
Sometimes we identify the interpretation $\mathbb A\rightsquigarrow\mathbb B$ with its code $\Gamma$.
We may also write $\mathbb A\stackrel{\Gamma}{\rightsquigarrow}\mathbb B$ or $\mathbb A\cong\Gamma(\mathbb B)$ to say that $\mathbb A$ is interpretable in $\mathbb B$ by means of $\Gamma$.

By $\mu_\Gamma$ we denote a surjective map $A_\Gamma \to A$ that gives rise to an isomorphism $\bar \mu_\Gamma\colon \Gamma(\MB) \to \MA$. We refer to $\mu_\Gamma$ as the {\em coordinate map} of the interpretation $\Gamma$. Note that such $\mu_\Gamma$ may not be unique because $\bar\mu_\Gamma$ is defined up to an automorphism of $\MA$. For this reason, the coordinate map $\mu_\Gamma$ is sometimes added to the interpretation notation: $(\Gamma, \mu_\Gamma)$. A coordinate map $\mu_\Gamma\colon A_\Gamma \to A$ gives rise to a map $\mu_\Gamma^m\colon A^m_\Gamma\to A^m$ on the corresponding Cartesian powers, which we often denote by $\mu_\Gamma$.

When the formula $E_\Gamma$ defines the identity relation $(x_1=x^\prime_1)\wedge\ldots\wedge(x_n=x^\prime_n)$, the surjection $\mu_\Gamma$ is injective. In this case, $(\Gamma,\bar p, \mu_\Gamma)$ is called an {\em injective interpretation}.




Of primary interest to us is the following well-known observation. 

\begin{theorem} [Lemma 4.4(2), \cite{KharlampovichMyasnikovSohrabi:2021}]\label{th:equiv0}
Let $\mathbb A$ be absolutely interpretable in $\mathbb B$ as $\mathbb A \cong \Gamma(\mathbb B)$. Then for any $\widetilde{\mathbb B}\equiv\mathbb B$ the algebraic structure $\Gamma(\widetilde{\mathbb B})$ is elementarily equivalent to $\mathbb A$.
\end{theorem}

\begin{remark}\label{re:relax}
Although outside of the scope of this paper, Theorem~\ref{th:equiv0} can be stated in greater generality, by relaxing the requirement of an absolute interpretation to a regular interpretation. We refer the reader to~\cite{Daniyarova-Myasnikov:I} for definitions and details.
\end{remark}

\begin{theorem} [Uniqueness of nonstandard models, \cite{Daniyarova-Myasnikov:I}]\label{th:equiv2}
Let a finitely generated structure $\mathbb A$ in a finite signature be absolutely interpretable in $\mathbb Z$ in two ways, as $\Gamma_1(\mathbb Z)$ and $\Gamma_2(\mathbb Z)$. Then there exists a formula $\theta(\bar x_1,\bar x_2)$, $|\bar x_i|=\dim\Gamma_i$, that defines an isomorphism $\Gamma_1(\mathbb Z)\to\Gamma_2(\mathbb Z)$. Moreover, if $\widetilde{\mathbb Z} \equiv \mathbb Z$, then $\theta(\bar x_1,\bar x_2)$ defines an isomorphism $\Gamma_1(\widetilde{\mathbb Z})\to\Gamma_2(\widetilde{\mathbb Z})$.
\end{theorem}

If $\MA$ is as in Theorem~\ref{th:equiv2} then the structure $\MA(\widetilde{\mathbb Z})$ is called the {\em nonstandard model of $\MA$ with respect to $\widetilde{\mathbb Z}$}. If $\MA \cong \Gamma(\Z)$ is an interpretation of $\MA$ in $\Z$ then $\MA(\widetilde{\mathbb Z}) \cong \Gamma(\widetilde{\mathbb Z})$ and the algebraic structure of $\MA(\widetilde{\mathbb Z})$ is revealed via the interpretation~$\Gamma$.

\subsection{Bi-interpretation}\label{se:bi-int}

\begin{definition}\label{def:reg}
Algebraic structures $\mathbb A$ and $\mathbb B$ are called {\em absolutely bi-interpretable} if 
\begin{enumerate}[1)]
\item there exist absolute interpretations $\mathbb A\cong \Gamma(\mathbb B)$ and $\mathbb B\cong \Delta(\mathbb A)$;
\item there exists a formula $\theta_\mathbb A(\bar u, x)$ in $L(\mathbb A)$, where $|\bar u|={\dim\Gamma\cdot\dim\Delta}$,  that defines some coordinate map $U_{\Gamma\circ\Delta}(\mathbb A)\to A$;
\item there exists formula $\theta_\mathbb B(\bar u, x)$ in $L(\mathbb B)$, where $|\bar u|={\dim\Gamma\cdot\dim\Delta}$, that defines some coordinate map $U_{\Delta\circ\Gamma}(\mathbb B)\to B$. 
\end{enumerate}
\end{definition}

Absolute bi-interpretability allows us to describe nonstandard models of one algebraic structure by means of another one. 

\begin{theorem}[\cite{Daniyarova-Myasnikov:I}]\label{th:equiv1}
Let $\mathbb A$ and $\mathbb B$ be absolutely bi-interpretable in each other, so $\mathbb A \cong \Gamma(\mathbb B)$ and $\mathbb B\cong\Delta(\mathbb A)$. Then
\begin{enumerate}[1)]
\item For any $\widetilde{\mathbb B}\equiv\mathbb B$ the algebraic structure $\Gamma(\widetilde{\mathbb B})$ is elementarily equivalent to $\mathbb A$;
\item Every $L(\mathbb A)$-structure $\widetilde{\mathbb A}$ elementarily equivalent to $\mathbb A$ is isomorphic to $\Gamma(\widetilde{\mathbb B})$ for some $\widetilde{\mathbb B}\equiv\mathbb B$;
\item For any $\mathbb B_1\equiv\mathbb B\equiv\mathbb B_2$ one has $\Gamma(\mathbb B_1) \cong \Gamma(\mathbb B_2) \iff \mathbb B_1 \cong \mathbb B_2.$
\end{enumerate}
\end{theorem}

\begin{remark}\label{re:relax-bi}
Similarly to Remark~\ref{re:relax}, Theorem~\ref{th:equiv1} can be stated in greater generality by relaxing the requirement of a bi-interpretation to a regularly invertible regular interpretation. We refer the reader to~\cite{Daniyarova-Myasnikov:I} for definitions and details.
\end{remark}

\section{Nonstandard tuples}\label{se:tuples}

\subsection{Nonstandard models of arithmetic}\label{se:nonstandard_arithmetic}

Here we introduce some well-known definitions and facts on nonstandard arithmetic. As a general reference we follow the book \cite{Kaye:1991}.

Let $\mathcal{L}_r = \{+,\cdot, 0,1\}$ be the language (signature) of rings with unity 1. By $\N = \langle N; +,\cdot,0,1\rangle$ we denote the \emph{standard arithmetic}, i.e., the set on non-negative integers $N$ with the standard addition $+$, standard multiplication $\cdot$, and constants $0, 1$. Sometimes, following common practice, we abuse the notation and denote the set $N$ as $\N$. Observe that $\langle N; +,\cdot,0,1\rangle$ is a semiring and $\langle N; +, 0\rangle$ is a commutative semigroup.

Usually by an arithmetic one understands any structure $\mathcal M$ in the signature $\mathcal{L}_r$ that satisfies Peano axioms.
However, for the purposes of the present paper, by a \emph{ model of arithmetic} we mean a structure $\nstdN$ in the signature $\mathcal{L}_r$ that is elementarily equivalent to $\mathbb N$, $\nstdN \equiv \N$. A model $\nstdM$ of arithmetic is called \emph{nonstandard} if it is not isomorphic to $\mathbb N$.
Further, notice that $\mathbb N$ is the only model of arithmetic finitely generated as an additive semigroup or as a semiring. In an equivalent approach one may use the ring of integers $\mathbb Z$ as the arithmetic, in this case nonstandard models of $\mathbb{Z}$ are exactly the rings $\widetilde{\mathbb{Z}}$ which are elementarily equivalent to $\mathbb{Z}$, but not isomorphic to $\Z$. Again, $\Z$ is the only (up to isomorphism) finitely generated ring which is elementarily equivalent to $\Z$.
These both approaches are indeed equivalent from the model theory viewpoint since $\N$ and $\Z$ are absolutely bi-interpretable in each other (in fact, $\mathbb{N}$ is definable in $\mathbb{Z}$ by Lagrange's four-square theorem). 

Every nonstandard model $\nstdM$ of $\mathbb N$, in terms of order, can be described as follows. Recall that for ordered disjoint sets $A$ and $B$ by $A + B$ we denote the set $A \cup B$ with the given orders on $A$ and $B$ and $a < b$ for every $a \in A$ and $b \in B$. Also by $A\cdot B$ we denote the Cartesian product $A\times B = \{(a,b) \mid a \in A, b \in B\}$ with the left lexicographical order, i.e., $(a_1,b_1) < (a_2,b_2)$ if and only if either $a_1 < a_2$ or $a_1 = a_2$ and $b_1 < b_2$. Then $\nstdM$, as a linear order, has the form $\mathbb N+Q_\lambda \mathbb Z$, for some dense linear order $Q_\lambda$ without endpoints of cardinality $\lambda = |\nstdN|$, that is $\mathbb N$ followed by $Q_\lambda$ copies of $\mathbb Z$, as shown in Fig.~\ref{fig:hairbrush}. Note that the ordered set $Q_\lambda$ is uniquely determined by $\nstdN$.
\begin{figure}[thb]
 \centering
 \includegraphics[width=0.5\linewidth]{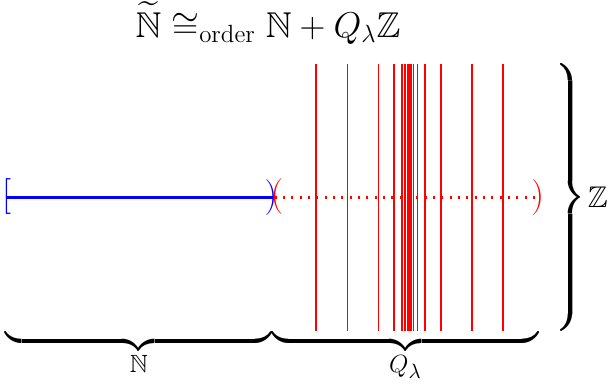}
 \caption{Nonstandard model of arithmetic.}
 \label{fig:hairbrush}
\end{figure}
In particular, any nonstandard countable model of arithmetic has the form $\mathbb N+\mathbb Q\mathbb Z$, since all countable dense linear ordered sets without endpoints are isomorphic.

In the countable case, while the order structure of $\nstdN$ is straightforward, each of the operations $+$ and $\cdot$ on $\nstdN$ is necessarily non-recursive~\cite{Tennenbaum}, so countable nonstandard models of arithmetic $\nstdN$ are non-constructible or non-recursive in the sense of Malcev~\cite{Maltsev:61} or Rabin~\cite{Rabin:60}, respectively.

One can identify a natural number $n \in \N$ with a nonstandard number $\tilde n = \tilde 1 + \ldots + \tilde 1$, which is the sum of $n$ nonstandard units $\tilde 1$ in $\nstdN$. The map $n \to \tilde n$ gives an elementary embedding $\N \to \nstdN$, i.e., with respect to this identification $\N$ is an elementary substructure of $\nstdN$. In other words, $\mathbb N$ is a prime model of the the theory $\Th(\N)$. In the sequel we always assume that $\mathbb N\subseteq \nstdM$ via this embedding. Elements of $\N$ in $\nstdN$ are called the \emph{standard natural numbers} in $\nstdN$. They form an initial segment of $\nstdN$ with respect to the order $<$, i.e., if $x\in\nstdM$ and $x<y$ for some $y\in \mathbb N$ then $x\in\mathbb N$. It follows that in every nonstandard model $\nstdM$, there is an element $x\in\nstdM$ s.t. $x>\mathbb N$ (the latter is a shorthand for $\forall y\in\mathbb N\ x>y$).

It is a crucial fact that $\mathbb N$ is not definable in $\nstdN$ (even by a formula with parameters from $\nstdN$). It follows from the Peano induction axiom (in fact, a scheme of axioms). Another useful observation is that the respective nonstandard ring of integers $\widetilde{Z}$ has additive group $\widetilde{Z}^+$ that is not cyclic, nor finitely generated.

\subsection{List superstructure}\label{se:superstructure} Recall the definition of a list superstructure of a structure $\mathbb A = \langle A;L \rangle$ (see \cite{Ashaev-Belyaev-Myasnikov:1993},\cite{Bauval:1985} \cite{KharlampovichMyasnikov:2018a}) with some small, mostly notational, changes. Let $O$ be a linear ordering with a fixed element $1 \in O$. For any $k \in O$ with $k \geq 1$ put $[1,k] = \{x \in O \mid 1\leq x \leq k\}$. Let $\List(A,O)$ be the set of all functions $s:[1,k]\to A$, where $k \in O, k \geq 1$. We call such functions $O$-\emph{tuples}, or $O$-\emph{lists}, of elements of $\mathbb A$, and write sometimes $s=(s_1,\ldots,s_k)$. If $O = \N$ then $\List(A,\N)$ is precisely the set of all finite tuples of elements of $A$. In this case we sometimes omit $\mathbb N$ from the notation and write $\List(A)$.

Now, the list superstructure $\mathbb S(\MA,\MN)$ is defined as the three-sorted structure 
\[
\mathbb S(\mathbb A,\mathbb N) = \langle \mathbb A, \List(A),\mathbb N; t(s,a,i), l(s)\rangle,
\]
where $\mathbb N$
is the standard arithmetic, $\List(A)$ is as above, $l\colon \List(A) \to \mathbb N$ is the length function, i.e., $l(s)$ is the length $k$ of a tuple $s=(s_1, \ldots,s_{k})\in \List(A)$, and $t(x,y,z)$ is a predicate on $\List(A) \times A\times \mathbb N$ such that $t(s,a,i)$ holds in $\mathbb S(\mathbb A,\mathbb N)$ if and only if $s = (s_1, \ldots,s_{k})\in \List(A), i \in \mathbb N, 1\leq i \leq k$, and $a = s_i \in A$.

Note that usually the list superstructure $\mathbb S(\mathbb A,\mathbb N)$ comes equipped (see the papers \cite{Ashaev-Belyaev-Myasnikov:1993},\cite{Bauval:1985} \cite{KharlampovichMyasnikov:2018a} mentioned above) with one more binary operation $\frown$ in the language, which is interpreted in $\mathbb S(\mathbb A,\mathbb N)$ as the concatenation of tuples, and an extra binary predicate $\in$ on $A\times \List(A)$ such that for $a \in A$ and $s \in \List(A)$, $a \in s$ holds in $\mathbb S(\mathbb A,\mathbb N)$ if and only if $a$ is the component $s_i$ of $s$ for some $i$. However, the predicate $\in$ and the concatenation $\frown$ are $0$-definable in $\mathbb S(\mathbb A,\mathbb N)$ (with the use of $t(s,i,a)$ and $\ell(s)$), so we omit them from the language, but sometimes use them in formulas, as convenient.

\begin{remark}\label{re:two_sorts}
In the case when $\mathbb A=\mathbb N$ one can simplify things by identifying $\mathbb A$ with $\mathbb N$ in $\mathbb S(\mathbb A,\mathbb N) = \langle \mathbb A, \List(A),\mathbb N; t(s,a,i), l(s)\rangle$ and use the two sorted structure $\mathbb S(\mathbb N,\mathbb N) = \langle \mathbb N, \List(\mathbb N); t(s,a,i), l(s)\rangle$ instead of the usual three-sorted one. Note that these structures are absolutely bi-interpretable in each other, so we can interchangeably use both structures. 
\end{remark}


\subsection{Interpretation of the list superstructure in arithmetic}\label{se:list_interpretation}

To define the nonstandard list superstructure in Section~\ref{se:nonstd_list} we need to describe an interpretation of~$\mathbb S(\mathbb N,\mathbb N)$ in~$\mathbb N$, through the well-known G\"odel enumeration.

To that end, we fix a computable enumeration of all finite tuples $\tau: \mathbb N\to \cup_{k=1}^\infty \mathbb N^k$. Denote the $i$-th component of $\tau(n)$ by $\tau(n)_i$:
\[
\tau(n)_i=a\iff \mathbb S(\mathbb N,\mathbb N)\models t(\tau(n),a,i).
\]
Then the partial function $f(n,i)=\tau(n)_i$ is computable, and therefore the set $\{(n,a,i)\mid a= f(n,i)\}\subseteq \mathbb N^3$ is recursively enumerable. Hence, there is a first order formula (in fact, a positive primitive, i.e., a Diophantine formula) $T_\tau(n,a,i)$ in $\mathbb N$ s.t.
\[
\tau(n)_i=a\iff \mathbb N\models T_\tau(n,a,i).
\]
Further, the length $k=\ell(\tau(n))$ of the tuple $\tau(n)$ is also computable, and therefore is given by a formula $L_\tau(n,k)$:
\[
\ell(\tau(n))=k\iff \mathbb N\models L_\tau(n,k).
\]
Intuitively, the formulas $T_\tau$ and $L_\tau$ allow one to think of each $n\in\mathbb N$ as corresponding to a tuple. In particular, $T_\tau$, $L_\tau$ are such that the following statements are true in $\mathbb N$:
\begin{align}
\tag{S1} &\label{eq:tuple1}\forall n\ \exists k\ L_\tau(n,k), \\
\tag{S2} &\label{eq:tuple2}\forall n\ \forall k\ \forall i \ [(L_\tau(n,k)\wedge i>k)\to \forall a\ \neg T_\tau(n,a,i)],\\
\tag{S3} &\label{eq:tuple3}\forall n\ \forall k\ \forall i\ [ (L_\tau(n,k)\wedge i\le k)\to \exists! a\ T_\tau(n,a,i)],\\
\tag{S4} &\label{eq:tuple4}\forall n,m\ [\forall i\ \forall a\ T_\tau(n,a,i)\leftrightarrow T_\tau(m,a,i)]\to m=n
\end{align}
With these formulas in mind, in the sequel we use expressions with $\ell$ and $t$ in first-order formulas with understanding that these expressions are to be replaced with formulas involving $L_\tau$ and $T_\tau$, respectively. For example, in place of~\eqref{eq:tuple3} we may write $\forall n\ \forall i\le \ell(\tau(n))\ \exists! a\ t(\tau(n),a,i)$.
As we mentioned above, concatenation $^\frown$ can be defined through $t$ and $\ell$, and therefore through $T$ and $L$. Namely, the following first-order formula states that for every $x,y\in\mathbb N$ there exists $z\in\mathbb N$ s.t. $\tau(z)=\tau(x)^\frown\tau(y)$:
\begin{equation}\label{eq:concat}
\begin{split}
 &\forall x,y\ \exists z\\
 &\ell(\tau(z))=\ell(\tau(x))+\ell(\tau(y))\\
 &\wedge [\forall a\ \forall i \le \ell(\tau(x))\ T_\tau(z,a,i)\leftrightarrow T_\tau(x,a,i) ]\\
 &\wedge [\forall a\ \forall \ell(\tau(x))< i \le \ell(\tau(x))+\ell(\tau(y))\\
 &T_\tau(z,a,i)\leftrightarrow T_\tau(y,a,i-\ell(\tau(x))) ].
\end{split}
\end{equation}

The next formula holds when and only when $a\in \tau(x)$:
\begin{equation}\label{eq:membership}
 \exists i\ 1\le i\le \ell(x) \wedge T(x,a,i).
\end{equation}

More generally, we can consider arithmetic enumerations $\tau$ (not necessarily the computable ones), that is, when $\tau$ is given by formulas $T$ and $L$ in the language of arithmetic that satisfy~\eqref{eq:tuple1}--\eqref{eq:tuple4}. In that event, for every $n\in\mathbb N$, the truth set of $T(n,a,i)$ defines a tuple corresponding to $n$. To make sure that every tuple in $\List(N,\mathbb N)$ corresponds to some tuple defined by $T(n,a,i)$ we add an infinite system of axioms. Namely, for each $k\in\mathbb N,$ we write
\begin{equation}\label{eq:tuple7}\tag{S5}
\forall a_1,\ldots,a_k\ \exists n\ L(n,k) \wedge \bigwedge_{i=1}^k T(n,a_i,i).
\end{equation}
Note that these formulas are satisfied whenever $L=L_\tau$ and $T=T_\tau$ for some computable enumeration $\tau$ as above.

\begin{definition}\label{de:arithmetic_enumeration}
The pair of formulas $\mathcal E=(T,L)$ that satisfy~\eqref{eq:tuple1}--\eqref{eq:tuple7} are called an \emph{arithmetical enumeration}
of tuples in $\mathbb N$. For an arithmetical enumeration~$\mathcal E$, by~$\mathcal E(n)$ we denote the tuple corresponding to $n\in\mathbb N$ under~$\mathcal E$.
\end{definition}

\begin{lemma}\label{le:arithmetical_enum_welldef}
Let $\mathcal E$ and $\mathcal D$ be two arithmetical enumerations of tuples in $\mathbb N$.
 There exists a definable bijection $f:\mathbb N\to\mathbb N$ s.t. for all $n_1,n_2\in\mathbb N$, $\mathcal E(n_1)=\mathcal D(n_2)$ if and only if $f(n_1)=n_2$.
\end{lemma}
\begin{proof} To prove the lemma it suffices to note that the set $\{(n_1,n_2)\in\mathbb N^2\mid \mathcal E(n_1)=\mathcal D(n_2)\}$ is definable in $\mathbb{N}$. 
\end{proof}
Now we explain how every arithmetical enumeration $\mathcal E$ defines an absolute interpretation of $\mathbb S(\mathbb N,\mathbb N)$ in $\mathbb N$,
\[
\mathbb S(\mathbb N,\mathbb N)\cong \Gamma_\mathcal E(\mathbb N).
\]


For simplicity we can represent the two-sorted structure $\mathbb S(\mathbb N,\mathbb N)$ (see Remark~\ref{re:two_sorts}) by the usual one-sorted one where the base set is a disjoint union of $\N$ and $\List(\N)$, equipped with the two unary predicates defining the subsets $\N$ and $\List(\N)$, two predicates representing the operations $+$ and $\cdot$ on $\N$, the predicate $t(s,a,i)$, and a predicate representing the length function $l(s)$. 

Now we follow Definition~\ref{de:interpretable} and notation there to describe $\Gamma_\varepsilon$. We interpret $\N \cup \List(\N)$ as subset of $\mathbb N^2$, where $\mathbb N^*=\mathbb N\times \{1\}$ and $\List(\mathbb N)^*=\mathbb N\times\{2\}$ (both are definable subsets of $\N^2$), the equality as the equivalence relation on $\mathbb N^* \cup \List(\N)^*$. The operations $+$ and $\cdot$ are interpreted naturally via the first component of $\N^*$. To interpret the predicate $t(s,a,i)$ on 
$\mathbb N^* \cup \List(\N)^*$ one uses the formula $T(n,a,i)$, and to interpret the function $l(s)$---the formula $L(n,k)$.

Formulas~\eqref{eq:tuple1}--\eqref{eq:tuple7} guarantee that the interpretation $\langle \mathbb N^*,\List(\mathbb N)^*; t^\mathbb N,\ell^\mathbb N\rangle$ is isomorphic to $\mathbb S(\mathbb N,\mathbb N)$. Indeed, the morphism is delivered by formulas \eqref{eq:tuple1}--\eqref{eq:tuple3}, and the bijectivity---by \eqref{eq:tuple4},\eqref{eq:tuple7}. 

If we were so inclined, it would be easy to define $\mathbb N^*$ and $\List(\mathbb N)^*$ in $\mathbb N$ instead of $\mathbb N^2$, say by placing the first sort in even integers, and the second sort in odd integers.

Note that the reverse absolute interpretation $\mathbb N\cong\Delta(\mathbb S(\mathbb N,\mathbb N))$ is constructed straightforwardly by interpreting $\mathbb N$ as the sort $\mathbb N$ of $\mathbb S(\mathbb N,\mathbb N)$.
\begin{proposition}[\cite{Cooper:2017,Rogers:1967}]\label{le:list_bi-int}
In the above notation the following holds:
\begin{itemize} \item [1)] $\mathbb S(\mathbb N,\mathbb N)\cong \Gamma_\mathcal E(\mathbb N)$ and the reverse interpretation $\mathbb N\cong \Delta(\mathbb S(\mathbb N,\mathbb N))$ give an absolute bi-interpretation of $\mathbb S(\mathbb N,\mathbb N)$ and $\mathbb N$ in each other.
\item [2)] Moreover, for two arithmetical enumerations $\mathcal E,\mathcal D$ the definable function $f$ from Lemma~\ref{le:arithmetical_enum_welldef} gives an isomorphism of interpretations $\Gamma_\mathcal E$ and $\Gamma_\mathcal D$.

\end{itemize}

\end{proposition}


In the sequel we will need tuples of tuples in $\mathbb N$, that is functions $[1,k]\to \List(\mathbb N,\mathbb N)$. An arithmetical enumeration $\mathcal E$ suffices to interpret $\mathbb S(\mathbb S(\mathbb N,\mathbb N),\mathbb N)$ in $\mathbb N$. Indeed, for each $n\in\mathbb N$, we consider the corresponding tuple $\mathcal E(n)=(n_1,\ldots,n_k)\in\mathbb S(\mathbb N,\mathbb N)$. To the same $n$, we then associate the tuple of tuples $(\mathcal E(n_i),\ldots,\mathcal E(n_k))$. We note that the length of $i$th tuple $\ell(\mathcal E(n_1))$ and $j$the component of $i$th tuple $t(\mathcal E(n_i),a,j)$ are definable using the original formulas $T$ and $L$. The following statement follows from this observation.

\begin{proposition}\label{le:tuples_tuples}
Let $\mathcal E$ be an arithmetical enumeration of tuples in $\mathbb{N}$. Then 
\begin{enumerate}[(a)]
    \item For every absolute interpretation $\Delta$ of a structure $\mathbb A$ in $\mathbb N$, there is a corresponding absolute interpretation $\Delta_\mathcal E$ of $\mathbb S(\mathbb A,\mathbb N)$ in $\mathbb N$.
    \item If, additionally, $\Delta$ gives rise to an absolute bi-interpretation, then so does $\Delta_\mathcal E$.
    \item If, additionally, $\mathcal D$ is an arithmetical enumeration of tuples, then $\Delta_\mathcal E$ and $\Delta_\mathcal D$ are definably isomorphic.
\end{enumerate}
\end{proposition}

\subsection{Nonstandard list superstructure}\label{se:nonstd_list}
Let $\mathcal E=(T,L)$ be an arithmetical enumeration of tuples in $\mathbb N$. Let $\nstdM\equiv\mathbb N$ be a nonstandard model of arithmetic. 

\begin{definition} Let $\mathcal E=(T,L)$ be an arithmetical enumeration of tuples. A function $s:[1,k]\to\nstdM$, $k\in\nstdM$, is called a \emph{nonstandard} (or \emph{definable}) tuple (list) if its graph is the truth set of $T(n,a,i)$ for some $n\in\nstdM$. By $\List_\mathcal E(\nstdM)$ we denote the the set of all such nonstandard tuples in $\nstdN$. 
\end{definition}

Note that $\List_\mathcal E(\mathbb N)=\List(\N,\mathbb N)$ for every arithmetical enumeration~$\mathcal E$. However, at least out of cardinality considerations, $\List_\mathcal E(\nstdN) \neq \List(\nstdN,\nstdN)$.

\begin{definition}
By Theorem~\ref{th:equiv1}, a structure is elementarily equivalent to $\mathbb S(\mathbb N,\mathbb N)$ when and only when it is of the form $\Gamma_\mathcal E(\nstdM)$ for some $\nstdM \equiv \mathbb N$, where $\Gamma_\mathcal E$ is as in Section~\ref{se:list_interpretation}. We denote $\Gamma_\mathcal E(\nstdN)$ by $\widetilde{\mathbb S}(\nstdN,\nstdN)$ and call it the \emph{nonstandard list superstructure of arithmetic $\nstdN$}.
\end{definition}
We will exploit the interpretation $\Gamma_\mathcal E$
and elementary equivalence to understand $\widetilde{\mathbb S}(\nstdN,\nstdN)$ in more specific terms.

The structure $\widetilde{\mathbb S}(\nstdN,\nstdN)$ is two-sorted with two underlying sets: $\nstdN$ and $\List_\mathcal E(\nstdM)$.
According to the interpretation $\Gamma_\mathcal E$, the set $\nstdN$ is interpreted as $\nstdN^* = \nstdN \times \{1\}$ and $\List_\mathcal E(\nstdM)$ as $\List_\mathcal E(\nstdM)^* = \nstdN \times \{2\}$. We can naturally identify $\nstdN^*$ and $\List_\mathcal E(\nstdM)^*$ with two disjoint copies of $\nstdN$.

After this identification each fixed $n \in \List_\mathcal E(\nstdM)^* = \nstdN$ defines uniquely the nonstandard tuple $s:[1,k] \to \nstdN$ where $s(i)=a$ if and only if $\nstdN\models L_\mathcal E(n,k)\wedge T_\mathcal E(n,a,i)$, which is equivalent to 
\[
\widetilde{\mathbb S}(\nstdN,\nstdN)\models [\ell(s)=k]\wedge [\forall 1\le i\le k\ t(s,s(i),i)].
\]
In this event, we denote $s=\mathcal E(n)$ and refer to it as the \emph{$n$th nonstandard tuple in $\nstdN$}. Notice that if $\nstdM$ is countable, then so is $\List_\mathcal E(\nstdM)$, unlike the set of \emph{all} tuples of elements of $\nstdN$ indexed by $1\le i\le n$, $n\in \nstdN$.

As usual, the element $a$ such that $\widetilde{\mathbb S}(\nstdN,\nstdN)\models t(s,a,i)$ is called the $i$th term, or $i$th component, of a nonstandard tuple $s$. To indicate a nonstandard tuple with the first element $a_1$ and the last element $a_k$, $k\in\nstdN$, we write $(a_1,\ndots,a_k)$.
Concatenation of nonstandard tuples is described by formula~\eqref{eq:concat}; membership $a\in (a_1,\ldots,a_k)$ is given by formula \eqref{eq:membership}.

\begin{proposition} Let $\mathcal E$ and $\mathcal F$ be arithmetical enumerations of tuples in $\N$ and $\nstdN \equiv \N$. Then the following holds:
\begin{itemize}
\item [1)] $\List_\mathcal E(\nstdN) = \List_\mathcal F(\nstdN)$ and there is a formula $\phi(x,y)$ in the language of rings $\mathcal L_r$ such that for every $m,n \in \nstdN$ 
\[
\mathcal E(m) = \mathcal F(n) \Longleftrightarrow 
\nstdN \models \phi(m,n)
\]

\item [2)] the interpretations $\Gamma_\mathcal E (\nstdN)$ and $\Gamma_\mathcal F (\nstdN)$ are isomorphic. Moreover, there is a formula $\psi(\bar x, \bar y)$ in the language of rings $\mathcal L_r$ that defines in $\nstdN$ an isomorphism $\Gamma_\mathcal E (\nstdN) \to \Gamma_\mathcal F (\nstdN)$.
\end{itemize}
\end{proposition}
\begin{proof}
Since $\widetilde{\mathbb S}(\nstdM,\nstdM)\equiv\mathbb S(\mathbb N,\mathbb N)$ the result follows from Lemma~\ref{le:arithmetical_enum_welldef}. 
\end{proof}
\begin{remark}
The set $\List_\mathcal E(\nstdN)$ and the interpretation $\Gamma_\mathcal E(\nstdN)$ do not depend (the latter up to a definable isomorphism) on the enumeration $\mathcal E$. We denote them 
$\List(\nstdN)$ and $\widetilde{\mathbb S}(\nstdN,\nstdN)$ respectively. \end{remark}

\begin{remark}
The structure $\widetilde{\mathbb S}(\nstdN,\mathbb N)$
 is \emph{not} elementarily equivalent to $\mathbb S(\mathbb N,\mathbb N)$, since such equivalence would give a way to define $\mathbb N$ in $\nstdN$. For example, $\mathbb N\subseteq \nstdN$ would be defined in $\widetilde{\mathbb S}(\nstdN,\mathbb N)$ by the formula
\[
\mathop{\mathrm{std}}(n)=\exists \mathrm{tuple}\, s\ ( s_1=1\wedge s_{\ell(s)}=n\wedge [\forall 1\le i\le\ell(s)-1\ s_{i+1}=s_{i}+1]).
\]
The sentence $\forall n\, \mathop{\mathrm{std}}(n)$ holds in $\mathbb S(\mathbb N,\mathbb N)$ but not in $\widetilde{\mathbb S}(\nstdN,\mathbb N)$. It follows that none of the nonstandard models $\nstdM$ have the same weak second order theory as $\mathbb N$. See also Section~\ref{se:summation_arb_superstructre} for a related discussion.
\end{remark}

\subsection{Properties of nonstandard tuples}\label{se:nonstd_tuples}

\subsubsection{Nonstandard tuples that are always there}
In this section we describe some tuples that belong to $\List(\nstdN)$ for every $\nstdN \equiv \N$.
\begin{lemma} [Definable functions] The following holds.
\begin{itemize}
\item [1)] Let $f:\N \to \N$ be an arithmetic function, i.e., there is a formula $\phi(x,y)$ of arithmetic such that $\forall m,n \in \N$ 
\[
f(m) = n \Longleftrightarrow \N \models \phi(m,n).
\]
Then for any $\nstdN \equiv \N$ and any $k \in \nstdN$ there is a nonstandard tuple $s = (s_1,\ldots,s_k) \in \List(\nstdN)$ such that $\forall i \leq k$ $\nstdN \models \phi(i,s_i)$.
\item [2)] Conversely, let $\phi(x,y)$ be a formula of arithmetic, such that for some $\nstdN \equiv \N$ and some infinite $k \in \nstdN$ there is a tuple $s \in \List(\nstdN)$ such that $l(s) = k$ and $\nstdN \models \phi(i,s_i)$ for every $i \leq k$. Then $\phi(x,y)$ defines an arithmetic function $f:\N \to \N$. 
\end{itemize}
\end{lemma}
\begin{proof}
Fix an arbitrary arithmetic enumeration $\mathcal E$ of tuples in $\N$.

For any $k \in \N$ there is a tuple $s \in \List(\N) = \List_\mathcal E(\N)$ such that $l(s) = k$ and $\forall i \leq k$ $\N \models \phi(i,s_i)$. It follows that 
\[
S(\nstdN,\nstdN) \models \forall k \in \N \ \exists s \in \List(\nstdN) [\forall i \leq k \ \phi(i,s_i)].
\]
Since 
\[
S(\nstdN,\nstdN) \cong \Gamma_\mathcal E(\nstdN) \equiv \Gamma_\mathcal E(\N) = \List_\mathcal E(\N)
\]
one has that for any $k \in \nstdN$ there is a nonstandard tuple $s \in \List_\mathcal E(\nstdN)$ such that $l(s) = k$ and $\nstdN \models \phi(i,s_i)$ for any $i \leq k$.
\end{proof}

The following sentence holds in $\mathbb S(\mathbb N,\mathbb N)$, hence in $\widetilde{\mathbb S}(\nstdN,\nstdN)$:
\[
\forall k\ \exists s\ \forall a\ \forall i\ t(s,a,i)\iff i\le k \wedge a=i.
\]
Therefore, the function $s$ defined by $s(i)=i$ for $1\le i\le k$ is a nonstandard tuple. We denote this nonstandard tuple by $[1,k]$ or $(1,\ndots,k)$.

Similar reasoning shows that there is a ``constant'' nonstandard tuple $(1,\ldots,1)$ of length $n$ for each $n\in\nstdM$.

For each nonstandard tuple $s$ and each $n\le \ell(s)$ we can consider its restriction to $[1,n]$, or its initial segment $s|_{n}$, defined by $\ell(s|_n)=n$ and $\forall i \forall a\ t(s|_{n},a,i)\leftrightarrow i\le n\wedge t(s,a,i)$.

\subsubsection{Nonstandard permutations}\label{se:nonstd_permutations} We say that a nonstandard tuple $\sigma$ of length $n=\ell(\sigma)\in\nstdN$ is a \emph{nonstandard permutation} of $[1,n]$ if each term of $\sigma$ is at least $1$ and at most $n$, and each $1\le k\le n$ appears exactly once in $\sigma$. The following first-order formula $P(\sigma,n)$ tells if a tuple $\sigma$ is a nonstandard permutation of $[1,n]$:
\begin{equation}\label{eq:permutation}
 \begin{split}
 P(\sigma,n)&=(\ell(\sigma)=n)\\
 &\wedge \forall 1\le k\le n\ \exists! 1\le i\le n\ t(\sigma,k,i)\\
 &\wedge \forall 1\le i\le n\ \exists! 1\le k\le n\ t(\sigma,k,i).
 \end{split}
\end{equation}
Note that every nonstandard permutation gives a first-order definable bijection of $[1,n]$ onto itself.

In the sequel, we often omit the word nonstandard and simply term such $\sigma$ \emph{permutations}. In particular, note that if $\sigma$ is a permutation, it is implied that it is a nonstandard tuple.

For illustration purposes, in the next proposition we show that the behavior we expect in the finite case is observed in the nonstandard case. In particular, note that it follows from this proposition that the second or the third line in~\eqref{eq:permutation} can be omitted.

\begin{proposition}
Suppose $\sigma$ is a nonstandard tuple of length $n\in\nstdN$.
\begin{enumerate}[(a)]
    \item If every $1\le i\le n$ appears in $\sigma$, then $\sigma$ is a permutation.
    \item If all components of $\sigma$ are distinct, and every component is at least $1$ and at most $n$, then $\sigma$ is a permutation.
\end{enumerate}
\end{proposition}
\begin{proof}
To show (a), note that in the finite case the formula
 \[
 (\ell(\sigma)=n)\; \wedge\; \forall 1\le k\le n\ \exists 1\le i\le n\ t(\sigma,k,i)
 \]
 implies~\eqref{eq:permutation}. By elementary equivalence, the same takes place for $\nstdN$ and nonstandard tuples. The reasoning for (b) is similar, with formula
 \[
 \begin{split}
 (\ell(\sigma)=n)\; &\wedge\; \forall 1\le i\le n\ \exists 1\le k\le n\ t(\sigma,k,i)\\
 &\wedge\; \forall 1\le i,j\le n\ \forall k\ t(\sigma,k,i)\wedge t(\sigma,k,j)\to i=j. 
 \end{split}
 \]
\end{proof}

\subsubsection{Nonstandard sum, product, and exponentiation}\label{se:nonstandard_summation} Using the above reasoning, we can extend sum, product and other definable functions to nonstandard tuples.
\begin{lemma}\label{le:summation}
Let $f:\mathbb N\times\mathbb N\to \mathbb N$ be a definable in $\mathbb N$ function. Then there is a definable in $\widetilde{\mathbb S}(\nstdM,\nstdM)$ function $\widetilde f:\widetilde{\mathbb S}(\nstdM,\nstdM)\to \nstdM$ s.t. $\widetilde f((a_1,a_2))=f(a_1,a_2)$ and $\widetilde f(a_1,\ldots,a_k)=f(\widetilde f((a_1,\ldots,a_{k-1})),a_k)$ for each $k\in\nstdM, k> 2$.
\end{lemma}
\begin{proof}
For a tuple $(a_1,\ldots,a_k)\in \List(\nstdN)$, we define a new tuple $s\in \List(\nstdN)$ of length $k$ by
\begin{enumerate}
    \item $s_1=a_1$,
    \item $s_{i+1}=f(s_i,a_{i+1})$, $1\le i\le k-1$.
\end{enumerate}
Then the function $(a_1,\ldots,a_k)\mapsto s_k$ is definable and satisfies the requirement of the lemma.
\end{proof}
We think of $\widetilde f$ as an extension of $f$ to the nonstandard case. For example, this allows summation over nonstandard tuples: for a tuple $s\in \widetilde{\mathbb S}(\nstdN,\nstdN),$ consider the tuple $r$ of the same length $\ell(s)=k$ given by
\begin{enumerate}
    \item $r_1=s_1$,
    \item $r_{i+1}=r_i+s_{i+1}$, $1\le i\le k-1$.
\end{enumerate}
Then the component $r_{k}$ gives the sum $\sum s=\sum_{i\le n} s_i$. We can similarly define $\prod s$. If $a,k\in \nstdN$ we can define $a^k$ as $\prod s$, where $\ell(s)=k$ and $s_i=a$ for each $1\le i\le k$.
It is easy to see (for example, by induction) that sum and product inherit associativity and commutativity in the following sense. If $s,s'$ are nonstandard tuples, then $\sum s+\sum s'=\sum s\frown s'$; if additionally $s'$ is a permutation of $s$, then $\sum s=\sum s'$. The same hold for the product.

As an example of commutativity and associativity, suppose $S_i$, $1\le i\le n$, are nonstandard tuples with $\ell(S_i)=m$ for each $1\le i\le n$, where $m,n\in\nstdM$, and $(i,j)\mapsto (S_i)_j$ is a definable function. Then we can observe that $\sum_i\sum_j (S_i)_j=\sum_j\sum_i (S_i)_j$. Indeed, the former sum by associativity is equal to $\sum_{k\le mn} T_{k}$, where $T_{im-m+j}=(S_i)_j$ for $1\le i\le n$, $1\le j\le m$, and the latter sum is equal to $\sum_{k\le mn}$, where $U_{jn-n+i}=(S_i)_j$ for $1\le i\le n$, $1\le j\le m$. $T$ and $U$ clearly differ by a nonstandard permutation, so by commutativity the two sums are equal.

We can also organize counting via the same approach: if $s$ is a nonstandard tuple of length $k\in\nstdN$, and $a\in \nstdN$, define $r$ as follows:
\begin{enumerate}
    \item If $s_1=a$, then $r_1=1$, and $r_1=0$ otherwise.
    \item For $1\le i\le k-1$, if $s_{i+1}=a$, then $r_{i+1}=r_{i}+1$, and $r_{i+1}=r_{i}$ otherwise.
\end{enumerate}
Then $r_k$ is the (nonstandard) number of occurrences of $a$ in $s$.
Equivalently, if we want to refer to Lemma~\ref{le:summation}, we first map $s\mapsto s'$ where $s'_i=1$ if $s_i=a$ and $s'_i=0$ otherwise (which is a definable mapping), and then take $\sum s'$.


\subsection{Nonstandard tuples over a given structure and weak second-order logic.}\label{se:summation_arb_superstructre}
The notion of nonstandard tuples can be generalized to an arbitrary structure $\mathbb B$, by considering an arbitrary $\mathcal M\equiv \mathbb S(\mathbb B,\mathbb N)$. However, in the present work we are only interested in nonstandard tuples over structures that can be interpreted in $\mathbb N$, for example, finitely presented groups. Below we describe such nonstandard tuples. We refer the reader to~\cite{Myasnikov-Nikolaev:2024} for the treatment of the case of an arbitrary $\mathbb B$.

Suppose $\mathbb A$ is interpretable in $\mathbb N$ (therefore, absolutely interpretable in $\mathbb N$).
Then since $\mathbb S(\mathbb N,\mathbb N)$ is interpretable in $\mathbb N$, so is $\mathbb S(\mathbb A, \mathbb N)$, say, $\mathbb S(\mathbb A, \mathbb N)\cong \Gamma(\mathbb N)$.
The list superstructure being 3-sorted, $\Gamma$ includes interpretation $\Gamma|_{\mathbb A}$ of $\mathbb A$ in $\mathbb N$ and interpretation $\Gamma|_{\mathbb N}$ of $\mathbb N$ in $\mathbb N$.
Then for every $\nstdN\equiv\mathbb N$, we consider $\Gamma(\nstdN)$. We denote $\widetilde{\mathbb A}=\Gamma|_{\mathbb A}(\mathbb A)$, observe $\nstdN\cong \Gamma|_{\mathbb N}(\nstdN)$, and denote $\Gamma(\nstdN)=\widetilde{\mathbb S}(\widetilde{\mathbb A},\nstdN)$.
Notice $\widetilde{\mathbb A}\equiv\mathbb A$ and $\widetilde{\mathbb S}(\widetilde{\mathbb A},\nstdN)\equiv \mathbb S(\mathbb A,\mathbb N)$ by Theorem~\ref{th:equiv0}. The structure $\widetilde{\mathbb S}(\widetilde{\mathbb A},\nstdN)$ is called a \emph{nonstandard list superstructure} of $\mathbb A$. For shortness, denote $\mathcal M=\widetilde{\mathbb S}(\widetilde{\mathbb A},\nstdN)$.

The nonstandard list superstructure $\mathcal M$ is three-sorted, with sorts $\widetilde{\mathbb A}\equiv \mathbb A$, $\nstdM\equiv \mathbb N$, and a sort $\widetilde S$ corresponding to $\List(\mathbb A,\mathbb N)$.
We refer to these sorts as the \emph{structure sort}, \emph{number sort}, and \emph{list sort} of $\mathcal M$, respectively.
We refer to the elements of the list sort $\widetilde{S}$ of $\mathcal M$ as \emph{nonstandard tuples over $\mathbb A$} and we treat them similarly to the nonstandard tuples in $\nstdN$ we described in Section~\ref{se:list_interpretation}.

Without loss of generality, we may assume that the elements of $\mathbb A$ are interpreted under $\Gamma$ as tuples over $\mathbb N$, that is, elements of $\List(\mathbb N,\mathbb N)$.
Then one can easily see that a function for $n\in\nstdN$, a function $s:[1,n]\to A$ represents a nonstandard tuple if and only if $(\ell(s(1)),\ldots,\ell(s(n)))$ and $s(1)\mathop{^\frown} s(2)\mathop{^\frown}\cdots\mathop{^\frown} s(n)$ are nonstandard tuples in $\nstdN$.
The latter observation gives a specific (modulo the notion of nonstandard tuples in $\nstdN$) description of nonstandard tuples in $\mathcal M$.
However, further details about nonstandard tuples over $\mathbb A$, exhibited below, do not explicitly rely on this description; indeed, they apply to elements of any structure $\mathcal M$ elementarily equivalent to $\mathbb S(\mathbb A,\mathbb N)$.
So, for the remainder of this section, we simply assume $\mathcal M$ to be an arbitrary such structure.

Notice that for each $a\in \widetilde S$ its length $\ell(s)$ is defined, and
\[
\mathcal M\models \forall a\in\widetilde S\ \forall 1\le i\le \ell(s)\ \exists! x\in A\ t(a,x,i),
\]
that is, each element $s$ of list sort of $\mathcal M$ has a unique $i$th component for each $1\le i\le \ell(s)$. Further, each $a$ is entirely defined by its components. Indeed, the following formula is satisfied in $\mathbb S(\mathbb A,\mathbb N)$ and therefore in $\mathcal M$:
\[
\forall a,b\ a=b \leftrightarrow [\forall x\in A\ \forall i\ (i\le \ell(a)\vee i\le \ell(b)) \to (t(a,x,i)\leftrightarrow t(b,x,i))].
\]
With that in mind, we often write $a=(a_1,\ldots,a_k)$, meaning that $\ell(a)=k$ and that predicates $t(a,a_1,1)$ and $t(a,a_k,k)$ are satisfied. We also write $a_i=x$ to denote that $t(a,x,i)$ is satisfied.

Like it is the case with standard tuples, we can define the membership predicate. Namely, we write $x\in a$ to denote the formula $\exists i\ t(a,x,i)$.

In the same vein, we can define the operation of concatenation $a\mathop{^\frown}b$ of $a,b\in \widetilde S$ through $t,\ell$ by taking advantage of elementary equivalence. To be specific, the following formula is satisfied in $\mathbb S(\mathbb A,\mathbb N)$ and therefore in $\mathcal M$:
\[
\begin{split}
\forall a,b\ &\exists! c\ \ell(c)=\ell(a)+\ell(b)\ \wedge\\
&\forall 1\le i\le\ell(c)\ [t(c,x,i)\leftrightarrow\\
&(i\le \ell(a)\wedge t(a,x,i))\vee (i>\ell(a) \wedge t(b,x,i-\ell(a)))].
\end{split}
\]
The unique $c$ given by the above formula is denoted $a\mathop{^\frown} b$, like in the case of standard tuples.

With the above onservations in mind, we say that a nonstandard tuple $a$ has length $k$ if $\ell(a)=k$, and we say that $x$ is the $i$th element of a nonstandard tuple $a$ if $\mathcal M\models t(a,x,i)$. Further, we can define the (nonstandard) number of occurrences $n_x(a)$ of $x$ in $a$, by defining first $n_x(a)$ to be the number of occurrences of $x$ in $a$ if $a$ has finite length, and then defining
\[
n_x(a_1,\ldots,a_k)=\begin{cases}
 n_x(a_1,\ldots,a_{k-1}),& x\neq a_k,\\
 n_x(a_1,\ldots,a_{k-1})+1,& x=a_k.
\end{cases}
\]
We say that $b$ is a nonstandard permutation of $a$ if $\forall x\in\widetilde{\mathbb A}\ n_x(a)=n_x(b)$. The same can be defined explicitly through nonstandard permutations introduced in~Subsection~\ref{se:nonstd_permutations}. Indeed, fix an arithmetic enumeration $\mathcal E$ (see Definition~\ref{de:arithmetic_enumeration}) and consider formula analogous to~\eqref{eq:permutation} that says that $n\in\mathbb N$ encodes a permutation under $\mathcal E$:
\[
\begin{split}
P(n)=\exists k\ L(n,k) &\wedge \forall 1\le i\le k\ \exists! 1\le j\le \ T(n,i,j)\\
&\wedge \forall 1\le j'\le k\ \exists! 1\le i'\le k\ T(n,i',j').
\end{split}
\]
Since $\mathcal M\equiv \mathbb S(\mathbb A,\mathbb N)$, it follows that $b$ is a nonstandard permutation of $a$ if and only if the following formula is satisfied in $\mathcal M$:
\[
\ell(a)=\ell(b) \wedge \exists n\ [P(n)\wedge L(n,\ell(a))\wedge \forall i,j\ T(n,i,j)\to b_j=a_i],
\]
where the latter equality $b_j=a_i$ is a shorthand for $\forall x\ t(b,x,j)\leftrightarrow t(a,x,i)$.

We can extend definable binary operations (for example, group product) on $\mathbb A$ to the tuples in $\mathcal M\equiv \mathbb S(\mathbb A,\mathbb N)$.

\begin{lemma}\label{le:arb_summation} Let $\mathbb A=\langle A;L\rangle$ be an arbitrary structure. Let $\mathcal M\equiv \mathbb S(\mathbb A,\mathbb N)$, with list sort $\widetilde{S}$. Let $f: A\times A\to A$ be a definable in $\mathbb A$ function. Then there is a definable in $\mathcal M$ function $\widetilde f:\widetilde{S}\to A$ s.t. $\widetilde f((a_1,a_2))=f(a_1,a_2)$ and $\widetilde f(a_1,\ldots,a_k)=f(\widetilde f((a_1,\ldots,a_{k-1})),a_k)$ for each $k\in\nstdM, k> 2$.
\end{lemma}
\begin{proof}
The proof essentially repeats that of Lemma~\ref{le:summation}.

Indeed, for a tuple $(a_1,\ldots,a_k)\in \widetilde S$, we define a new tuple $s\in \widetilde S$ of length $k$ by
\begin{enumerate}
    \item $s_1=a_1$,
    \item $s_{i+1}=f(s_i,a_{i+1})$, $1\le i\le k-1$.
\end{enumerate}
Then the function $(a_1,\ldots,a_k)\mapsto s_k$ is definable and satisfies the requirement of the lemma.
\end{proof}


Similarly to and in the same sense as in Subsection~\ref{se:nonstandard_summation}, properties like associativity and commutativity of $f$ carry to $\widetilde{f}$.
Namely, let $(b_1,\ldots,b_k)\in \widetilde{S}$ be a nonstandard permutation of $(a_1,\ldots,a_k)\in \widetilde{S}$. Then, if $f$ is associative and commutative, then $\widetilde f((a_{1},\ldots,a_{k}))=\widetilde f((b_1,\ldots,b_k))$.

\begin{remark}\label{re:terms_generalization}
Lemma~\ref{le:arb_summation} is a particular application of the following general observation. Suppose we have an indexed by $i\in\mathbb N$ family of functions $f_i$ in variables $x_1,\ldots,x_{k_i}$ defined by a formula of $\mathbb S(\mathbb A,\mathbb N)$ in the following sense: there is a formula $\sigma(z,\bar x,i)$ s.t. $\mathbb S(\mathbb A,\mathbb N)\models \sigma(b,\bar a,i)$ if and only if $\ell(\bar a)=k_i$ and $b=f_i(a_1,\ldots,a_{k_i})$.
Then over any $\mathcal M\equiv \mathbb S(\mathbb A,\mathbb N)$ with numbers sort~$\nstdN$, the same formula defines generalization of $f_i$ to $i\in\nstdN$, in particular, to $i>\mathbb N$.
\end{remark}


\subsubsection{Weak second-order logic and nonstandard models}\label{se:wso_translation}
By a weak second-order logic formula over a structure $\mathbb A=\langle A; L\rangle$ we mean a first-order formula over the list superstructure $\mathbb S(\mathbb A,\mathbb N)$. There are other, equivalent, descriptions, for details of which we refer the reader to~\cite{KharlampovichMyasnikovSohrabi:2021}.

For our investigation, it is a crucial observation that quantification over tuples becomes quantification over nonstandard tuples when a nonstandard model is considered. To formulate this, we start with the following observation.

\begin{proposition}\label{pr:list_interpretation}
If a structure $\mathbb A$ is absolutely interpretable in $\mathbb N$, then so is its list superstructure $\mathbb S(\mathbb A,\mathbb N)$.
\end{proposition}
\begin{proof}
Let $\mathbb A=\langle A;L\rangle$. Notice that $\mathbb S(\mathbb A,\mathbb N)$ can be interpreted in $\mathbb S(\mathbb N,\mathbb N)$. Indeed, a tuple $(a_1,\ldots,a_k)\in\List(A,\mathbb N)$ is interpreted as a tuple $a_1^\diamond\mathop{\frown}a_2^\diamond\mathop{\frown}\ldots\mathop{\frown}a_k^\diamond\in\List(\mathbb N,\mathbb N)$, where for each $1\le i\le k$, $a_i^\diamond$ is any representative of the equivalence class interpreting $a_i$. The rest of the interpretation is defined accordingly in a straightforward manner.

Since $\mathbb S(\mathbb N,\mathbb N)$ is absolutely interpretable in $\mathbb N$, the statement follows by composing the interpretations.
\end{proof}

Now, let $\mathbb A$ be absolutely interpretable in $\mathbb N$ as $\mathbb A\cong \Gamma(\mathbb N)$ and let $\mathbb S(\mathbb A,\mathbb N)\cong \Delta(\mathbb N)$ be the absolute interpretation provided by Proposition~\ref{pr:list_interpretation}. Suppose $\nstdN\equiv \mathbb N$ is a nonstandard model of $\mathbb N$ and consider $\widetilde{\mathbb A}\cong\Gamma(\nstdN)$ and $\mathcal M=\widetilde{\mathbb S}(\widetilde{\mathbb A},\nstdN)\cong \Delta(\mathbb N)$.

\begin{remark}\label{re:nonstd_tuples} Consider $\varphi$, a weak second-order formula of $\mathbb A$, that is, a first-order formula of $\mathbb S(\mathbb A,\mathbb N)$. If $\mathbb S(\mathbb A,\mathbb N)\models \varphi$, then $\mathcal M\models \varphi$ due to elementary equivalence. In particular, if $\varphi$ is of the form
\[
\forall \bar a\in \List(A,\mathbb N)\ \psi(\bar a)\mbox{\quad or\quad}\exists \bar a\in \List(A,\mathbb N)\ \psi(\bar a),
\]
then for $\mathcal M$ the same formula takes form
\[
\forall \bar a\in \List_{\widetilde{\mathbb A}}\ \psi(\bar a)\mbox{\quad or\quad}\exists \bar a\in \List_{\widetilde{\mathbb A}}\ \psi(\bar a),
\]
where $\List_{\widetilde{\mathbb A}}$ denotes the list sort of $\widetilde{\mathbb S}(\widetilde{\mathbb A}, \nstdN)$.
The same can be applied to the first quantifier of $\psi$, and so on. 
\end{remark}

In other words, under these assumptions, a statement of the weak second-order theory of $\mathbb A$ translates to a statement of the second-order theory of $\mathbb A$ by replacing quantification over tuples with quantification over nonstandard tuples.

\section{Free group}\label{se:free_group}
In this section, we introduce nonstandard models of a free group. As a byproduct, we obtain a description of ultrapowers of a free group. Let $F$ be a free group on $k$ generators $x_1,\ldots, x_k$. After Kharlampovich, Myasnikov~\cite{KharlampovichMyasnikov:2018b}, we interpret $F$ in the list superstructure $\mathbb S(\mathbb Z,\mathbb N)$, by setting a reduced group word $x_{i_1}^{\varepsilon_1}\cdots x_{i_n}^{\varepsilon_n}$, where $x_{i_j}^{\varepsilon_j}\neq (x_{i_{j+1}}^{\varepsilon_{j+1}})^{-1}$, to correspond to the tuple $(\varepsilon_1 i_1,\ldots, \varepsilon_n i_n)$, where $\varepsilon_ji_j\neq -\varepsilon_{j+1}i_{j+1}$. It is straightforward to see that the multiplication can be interpreted as concatenation of the respective tuples and removal of the longest tuple of the form $(d_\ell,\ldots,d_1,-d_1,\ldots,-d_\ell)$ symmetric about the concatenation point. This leads to an absolute interpretation $F\cong \Gamma(\mathbb S(\mathbb Z,\mathbb N)$.

This allows to give a family of groups elementarily equivalent to $F$.
Let $\nstdN\equiv\mathbb N$ be any nonstandard model of $\mathbb N$, $\widetilde{\mathbb Z}$---the corresponding model of $\mathbb Z$, and $\widetilde{\mathbb S}(\widetilde{\mathbb Z},\nstdN)$---the corresponding model of $\mathbb S(\mathbb Z,\mathbb N)$.
Then $\widetilde{F}\cong \Gamma(\widetilde{\mathbb S}(\widetilde{\mathbb Z},\nstdN))$ is elementarily equivalent to $F$ by~\cite[Lemma 4.4(2)]{KharlampovichMyasnikovSohrabi:2021}.
Denote $\widetilde{F}=F(\nstdN)$ (notice that this notation is well-defined by Theorem~\ref{th:equiv2}).

Now we fix a particular nonprincipal ultrafilter $D$ on $\omega$ and consider the corresponding ultrapowers $F^*\cong \prod\limits_{i\in \omega}F\ /D$ and $\mathbb Z^*\cong \prod\limits_{i\in \omega}\mathbb Z\ /D$. Since $\mathbb Z^*\equiv \mathbb Z$, we denote it by $\widetilde{\mathbb Z}$, as above.

\begin{theorem}\label{th:free_group_ultrapower}
Suppose $2^\omega=\omega+$. In the above notation, $F^*\cong \widetilde F$.
\end{theorem}
\begin{proof}
We note that the ultrapower $F^*$ is $\omega+$-saturated (therefore, saturated by continuum hypothesis $2^\omega=\omega+$), and so is $\mathbb Z$. It is straightforward to see that if $\mathbb A$ is saturated, then so is any structure absolutely interpreted in $\mathbb A$. Therefore, $\widetilde{F}\cong \Gamma(\widetilde{\mathbb Z})$ is saturated. Now, $F^*\cong \widetilde{F}$ since they are saturated models of the same cardinality (see, for example,~\cite[Theorem 4.3.20]{Marker}).
\end{proof}

In fact, a similar statement can be made for the case of the generalized continuum-hypothesis.
\begin{theorem}\label{th:free_group_ultrapower_gen}
Let $F$ be a finitely generated free group. Suppose $\kappa$ is an infinite cardinal such that $2^\kappa=\kappa+$. Let $D$ be a non-principal ultrafilter on $\kappa$ and $F^*\cong \prod\limits_{i\in \omega}F\ /D$ and $\mathbb Z^*\cong \prod\limits_{i\in \omega}\mathbb Z\ /D$ be the corresponding ultrapowers. Then $F^*\cong \widetilde F$, where $\widetilde{F}$ denotes the nonstandard model $F(\mathbb Z^*)$.
\end{theorem}
\begin{proof}
The proof repeats that of Theorem~\ref{th:free_group_ultrapower} with $\kappa$ instead of $\omega$.
\end{proof}

The same approach works to describe ultrapowers of arbitrary groups (or indeed algebraic structures) interpretable in $\mathbb N$, in particular, the groups whose word problem is arithmetic under a G\"odel enumeration. We elaborate on this in Section~\ref{se:presentations}.

\begin{remark}\label{re:iso}
    Under the same assumptions, we have $\widetilde{F_2}\cong F_2^*\cong F_3^*\cong \widetilde{F_3}$. It would be interesting to understand anything about the isomorphism $\widetilde{F_2}\cong \widetilde{F_3}$. In particular, in Remark~\ref{re:lovecraft} we show that this isomorphism is not weak second-order definable.
\end{remark}


We can describe the elements and operations of $\widetilde F$ explicitly, in terms of $\widetilde{\mathbb S}(\widetilde{\mathbb Z},\nstdN)$. Indeed, inspecting the interpretation $\Gamma$ above, we see that the elements of $\widetilde F$ are given by the nonstandard tuples $(a_1,a_2,\ldots,a_n)$, where $1\le |a_i|\le k$ for each $1\le i\le n$, and $a_i\neq -a_{i+1}$ for all $1\le i\le n-1$. The group multiplication is then given by concatenating the respective tuples and removing the maximal ``canceling out'' segment. Namely, let
\[
(a_1,\ldots,a_n,b_1,\ldots,b_m)=(a_1,\ldots,a_{n-\ell},d_{\ell},\ldots,d_1,-d_1,\ldots,-d_\ell,b_{\ell+1},\ldots, b_m),
\]
with $a_{n-\ell}\neq -b_{\ell+1}$, or $\ell =n$, or $\ell =m$. Then
\[
(a_1,\ldots,a_n)\cdot (b_1,\ldots,b_m)=(a_1,\ldots,a_{n-\ell},b_{\ell+1},\ldots, b_m).
\]
Group inversion is given by
\begin{equation}\label{eq:inversion}
(a_1,\ldots,a_n)^{-1}=(-a_n,\ldots,-a_1).
\end{equation}

In our interpretation, instead of setting group elements to correspond to ``reduced'' tuples, we can set them to correspond to equivalence classes of tuples under free cancellation. The resulting group, which for the moment we denote by $\mathop{\widetilde F}'$ is isomorphic to $\widetilde F$ by Theorem~\ref{th:equiv2}. We give the details here, since we take advantage of this approach later when we introduce nonstandard group presentations.

We say that two nonstandard tuples $(a_1,\ldots, a_n)$, $1\le |a_i|\le k$ for each $1\le i\le n$, and $(b_1,\ldots, b_m)$, $1\le |b_i|\le k$ for each $1\le i\le m$, are freely equivalent if there is a nonstandard tuple of nonstandard tuples
\[
\begin{split}
((a_1,\ldots,a_n)=(c_1^1,\ldots,c_{n_1}^1),&\\
(c_1^2,\ldots,c_{n_2}^2),&\\
\ldots,&\\
(b_1,\ldots,b_m)=(c_1^r,\ldots,c_{n_r}^r)&)\\
\end{split}
\]
such that for each $1\le i\le r-1$, the tuples $c^i$ and $c^{i+1}$ differ by deletion or insertion of a freely cancelling pair:
\[
\begin{split}
(c_1^i,\ldots,c_{n_i}^i)=(c_1^{i+1},\ldots,c_j^{i+1}, a, -a,c_{j+1}^{i+1}, \ldots, c_{n_{i+1}}^{i+1})&\quad \mbox{or}\\
(c_1^i,\ldots,c_j^{i}, a, -a,c_{j+1}^{i},\ldots,c_{n_i}^i)=(c_1^{i+1}, \ldots, c_{n_{i+1}}^{i+1})&,
\end{split}
\]
where $1\le |a|\le k$. Group multiplication and inversion are then interpreted in a straightforward way, multiplication by concatenation of representatives of equivalence classes, and inversion by~\eqref{eq:inversion} applied to a representative. Each equivalence class has a unique freely reduced representative, the above operations are well defined, and the reduced representative mapping $\mathop{\widetilde F}'\to \widetilde F$ is a group isomorphism. Indeed, all of these statements are true in the case of a free group $F$, and can be expressed by first-order formulas of $\mathbb S(\mathbb Z,\mathbb N)$. Therefore the same formulas hold for the elementarily equivalent $\widetilde{\mathbb S}(\widetilde{\mathbb Z},\nstdN)$, which translates to the above statements.

\section{Structure and basic properties of nonstandard free group}\label{se:basic_structure}
In this section we give examine basic properties of a nonstandard free group $\widetilde{F}\cong \Gamma(\widetilde{\mathbb S}(\widetilde{\mathbb Z}, \nstdN))$, where $F$ is a free group of rank $r$ freely generated by $x_1,\ldots,x_r$ and $\Gamma$ is as introduced in Section~\ref{se:free_group}.
The general idea is that, due to Remark~\ref{re:nonstd_tuples}, properties of the standard free group $F$ carry over to $\widetilde{F}$, with the change that whenever a finite sequence of elements is involved, the corresponding statement for $\widetilde{F}$ will utilize a nonstandard sequence.

For example, $\widetilde{F}$ is not finitely generated but it is \emph{nonstandardly} generated in the sense that for every $f\in\widetilde{F}$, there is a $k\in\nstdN$ and nonstandard tuples $(i_1,\ldots,i_k)$, $1\le i_j\le r$, and $(\varepsilon_1,\ldots,\varepsilon_k)$, $\varepsilon_j\in\{-1,1\}$, s.t. $\prod_j x_{i_j}^{\varepsilon_j}=f$.
Here and everywhere below the product is understood in the sense of Lemma~\ref{le:arb_summation}.
While this follows by definition of $\Gamma$, we take an opportunity to explain the same through Proposition~\ref{pr:list_interpretation} and Remark~\ref{re:nonstd_tuples}, which will be our main tool in this section.
Indeed, consider the formula $\Phi_{r}$ of the first order theory of $\mathbb S(F,\mathbb N)$ that records that there are $r$ elements $x_1,\ldots, x_r$ s.t. every element of $F$ can be expressed as a finite group product of elements of $\bar x = (x_1,\ldots,x_r)$ and their inverses:
\[
\begin{split}
Gen(\bar x) = \forall x\in F\ \exists \bar y\ \left(\forall 1\le i\le \ell(\bar y)\ \ y_i\in \bar x \vee y_i^{-1}\in \bar x\right) \wedge x=\prod\bar y.
\end{split}
\]
Since $S(F,\mathbb N)\models \exists\bar x\ (Gen(\bar x)\wedge \ell(\bar x)=r)$ for some $r$, we have $\widetilde{\mathbb S}(\widetilde{\mathbb Z}, \nstdN)\models \exists\bar x\ (Gen(\bar x)\wedge \ell(\bar x)=r)$.
The latter is the statement that every element $x$ of $\widetilde{F}$ can be expressed as product of elements of a tuple $(y_1,\ldots,y_k)$, where each $y_j$ is some $x_\ell^{\pm 1}$.
The tuple $(y_1,\ldots,y_k)$ is a member of the list sort of $\widetilde{\mathbb S}(\widetilde{F}, \nstdN)$, i.e., a nonstandard tuple over $\widetilde{F}$. In other words, every element of $\widetilde{F}$ is a nonstandard product of free generators of $\widetilde{F}$.

Consider formulas $Red(\bar y)$:
\[
Red(\bar y) = \forall 1\le i\le \ell(\bar y)\ \ y_i\neq y_{i+1}^{-1}
\]
and $Free(\bar x)$:
\[
\begin{split}
Free(\bar x)=\forall \bar y, \bar z\  \Bigl[&\left(\forall 1\le i\le \ell(\bar y)\ \ y_i\in \bar x\vee y_i^{-1}\in \bar x\right)\wedge Red(\bar y)\wedge \\
&\left(\forall 1\le j\le \ell(\bar z)\ \ z_j\in \bar x\vee z_j^{-1}\in \bar x\right)\wedge Red(\bar z)\wedge\\
& \prod\bar y=\prod\bar z\ \Bigr]\to \bar y =\bar z.
\end{split}
\]
Now we can record that a group is freely generated by some tuple $\bar x$:
\[
FreeGen(\bar x)= Gen(\bar x)\wedge Free(\bar x),
\]
and the formula $\Phi$ that records well-definedness of rank for a free group:
\[
\forall \bar x,\bar x'\ \ FreeGen(\bar x)\wedge FreeGen(\bar x') \to \ell(\bar x)=\ell(\bar x').
\]
\begin{remark}\label{re:lovecraft}
    It follows that the isomorphism in Remark~\ref{re:iso} is not definable in $\widetilde{\mathbb S}(\widetilde{\mathbb Z},\nstdN)$ and does not respect nonstandard products.
\end{remark}

In the standard free group, two products of free generators are equal if and only if there is a sequence of free cancellations and free insertions that transforms one product into the other. Below is an analogue of this statement for a nonstandard free group.

\begin{proposition}
Suppose $\bar x$ is a tuple freely generating a nonstandard free group $\widetilde{F}$. Suppose $\bar a$ and $\bar b$ are two nonstandard tuples of elements of $\bar x$ and their inverses such that $\prod \bar a=\prod \bar b$. Then there is a nonstandard tuple of nonstandard tuples $(\bar a= \bar c^1,\ldots, \bar b = \bar c^r)$ s.t. $\bar c^i$ and $\bar c^{i+1}$ differ by a single free cancellation or insertion.
\end{proposition}
\begin{proof}
    Notice that the statement can be written as a formula $\varphi$ of $\widetilde{\mathbb S}(\widetilde{F}, \nstdN)$.
    Since $\mathbb S(F,\mathbb N)\models \varphi$, it follows $\widetilde{\mathbb S}(\widetilde{F}, \nstdN)\models \varphi$.
\end{proof}

\subsection{Subgroups of nonstandard free group.}\label{se:nstd_subgroups}
For our use in the remaining part of the paper, we establish terminology concerned with subgroups and their generators.
Finitely generated subgroups could be generalized to the nonstandard case in four ways.
Informally speaking, we can consider finite or nonstandard generating sets, and we can consider finite or nonstandard products of elements of generating sets.
Each of these four options results in a subgroup of $\widetilde{F}$, however, not all of them can viewed as natural generalizations of finite subgroups.

By a \emph{standard} subgroup $G$ of $\widetilde{F}$ we mean a subgroup in the usual sense, that is, a subset $G\subseteq \widetilde{F}$ closed under group inversion and group product.
By a \emph{nonstandard} subgroup of $\widetilde{F}$ we mean a standard subgroup that is additionally closed under nonstandard products.
We view the latter object as the natural one in the context of this paper, and whenever we say subgroup in the sequel, we mean the nonstandard subgroup.
Correspondingly, the notation $G\le \widetilde{F}$ indicates that $G$ is a nonstandard subgroup.
Crucially, observe that if $G$ is a standard subgroup of $\widetilde{F}$ and is definable in $\widetilde{\mathbb S}(\widetilde{\mathbb Z}, \nstdN)$ (or in $\widetilde{\mathbb S}(\widetilde{F},\nstdN)$), then it is a nonstandard subgroup.

Consider a set $U=\{u_\alpha:\alpha\in A\}$ of elements of $\widetilde{F}$. For example, elements of $U$ can be enumerated by a tuple (finite or infinite nonstandard) $\bar u=(u_1,\ldots,u_k)$, $k\in\nstdN$, of elements of $\widetilde{F}$. In that case, with abuse of terminology, we do not distinguish between sets and tuples and write $U=\bar u$. $U$ can also be a set whose elements are not organized in a tuple, for example $U=\{x_2^{-j}x_1x_2^j:j\in\nstdN\}$.

For each $U$ as above, we can consider two subgroups:
\[
\begin{split}
G_s &= \langle U\rangle_s=\left\{u_{i_1}^{\varepsilon_1}\cdots u_{i_m}^{\varepsilon_m} \mid m\in\mathbb N, u_{i_j}\in U, \varepsilon_i=\pm 1\right\}\mbox{ and }\\
G_n &= \langle U\rangle_n=\left\{u_{i_1}^{\varepsilon_1}\cdots u_{i_m}^{\varepsilon_m} \mid m\in\nstdN, u_{i_j}\in U, \varepsilon_i=\pm 1\right\}.
\end{split}
\]
Notice that $G_s$ is the subgroup generated by the elements of $U$ in the usual sense, that is, the minimal subgroup that contains the elements of $U$.
We refer to $G_s$ as the standard subgroup generated by the elements of $U$.
Respectively, $G_n$ is the minimal nonstandard subgroup generated the elements of $U$. Notice that in the case when $U=\bar u$ is a tuple (finite or nonstandard) or, more generally, when $U$ is definable in $\widetilde{\mathbb S}(\widetilde{\mathbb Z}, \nstdN)$, $G_n$ is the minimal \emph{definable in $\widetilde{\mathbb S}(\widetilde{\mathbb Z}, \nstdN)$} subgroup of $\widetilde{F}$ that contains the elements of $\bar u$.
We refer to $G_n$ as the nonstandard subgroup generated by the elements of $\bar u$.
In the sequel, the notation $G=\langle U\rangle$ (or $G=\langle u_1,\ldots,u_k\rangle$ if $U=\bar u$) means the nonstandard subgroup $G=G_n$.

If $G$ is a subgroup of $\widetilde{F}$ and there is a tuple $\bar u=(u_1,\ldots,u_k)$ s.t. $G=\langle u_1,\ldots,u_k\rangle$ (with all of the above conventions), we say that $G$ is \emph{finitely generated} if $k\in\mathbb N$, and \emph{nonstandardly generated} if $k\in\nstdN$ (in particular, notice that a finitely generated subgroup is also nonstandardly generated). For example, it is easy to see that a subgroup generated by all elements of the form $x_2^{-j}x_1x_2^j$, $j\in\nstdN$, is not a nonstandardly generated subgroup.

\subsection{Properties of nonstandard free group.}
Notice that all of the following can be expressed by a weak second-order formula of $F$ (therefore, explicating interpretation, by a weak second-order formula of $\mathbb Z$):
\begin{enumerate}
    \item Elements $x_1,\ldots,x_r\in F$ freely generating $F$.
    \item The statement that two group words $w_1,w_2$ in free generators of $F$ are equal as elements of $F$ if and only if there is a fine sequence of free cancellation or insertions that takes $w_1$ to $w_2$.
    \item $y_1,\ldots,y_n$ being free generators of $\langle y_1,\ldots,y_n\rangle$.
    \item Membership of $x$ in a subgroup generated by $y_1,\ldots,y_k$.
    \item The statement that every finitely generated subgroup of $F$ is free.
    \item The statement that every finitely generated subgroup of $F$ has a Nielsen set of generators.
    \item The statement that $F$ possesses a Lyndon length function.
    \item Howson property, that is the statement that intersection of two finitely generated subgroups is finitely generated.
    \item Finite (nonstandard) index.
    \item Greenberg--Stallings theorem, that is, the statement that every finitely generated subgroup of $F$ is of finite index in its commensurator.
\end{enumerate}
(As a side note, we mention that in a way, all ``reasonable'' properties of a structure are weak second-order properties.)
All of the above have nonstandard analogues for $\widetilde{F}$.
We refrain from detailing every item on the above list.
We give two illustrations: we inspect centalizers in $\widetilde{F}$ and we offer a nonstandard version of the Howson property.

\begin{proposition}\label{pr:centralizers}
    Centralizers in $\widetilde{F}$ are isomorphic to the additive group~$\widetilde{\mathbb Z}^+$.
\end{proposition}
\begin{proof}
    Consider the formula $\varphi$ of $\mathbb S(F,\mathbb N)$ given by
    \[
    \forall x\in F\ \exists y\in F\  \forall z\in F\  \exists k\in\mathbb Z\quad [x,z]=1\to z=y^k.
    \]
    We have $\mathbb S(F,\mathbb N)\models \varphi$ since centalizers in $F$ are cyclic. Then $\widetilde{\mathbb S}(\widetilde{F},\nstdN)\models \varphi$, and for $\widetilde{F}$, by Remark~\ref{re:nonstd_tuples}, the formual $\varphi$ takes form
    \[
    \forall x\in \widetilde{F}\ \exists y\in \widetilde{F}\  \forall z\in \widetilde{F}\  \exists k\in\widetilde{\mathbb Z}\quad [x,z]=1\to z=y^k,
    \]
    in other words, centralizer of $x$ in $\widetilde{F}$ is $\{y^k\mid k\in\widetilde{\mathbb Z}\}\cong \widetilde{\mathbb Z}^+$.    
\end{proof}

For Howson property, let $\varphi_{\mathrm{memb}}(z, \bar y)$ be a formula of $\mathbb S(F,\mathbb N)$ s.t. $\mathbb S(F,N)\models \varphi_{\mathrm{memb}}(g, \bar g)$ if and only if $g\in \langle g_1,\ldots,g_m\rangle$, where $\bar g=(g_1,\ldots,g_m)$.
Then the Howson property is recorded by formula $\Phi_{\mathrm{H}}$ as follows:
\[
\Phi_{\mathrm{H}}=\forall \bar y^{(1)}, \bar y^{(2)}\
\exists \bar y\ \forall x\in F\ \varphi_{\mathrm{memb}}(x, \bar y^{(1)})\wedge \varphi_{\mathrm{memb}}(x, \bar y^{(2)})\leftrightarrow \varphi_{\mathrm{memb}}(x, \bar y).
\]

For a nonstandard model $\widetilde{F}$ of $F$ we have, by Proposition~\ref{pr:list_interpretation} and Remark~\ref{re:nonstd_tuples} that $\widetilde{\mathbb S}(\widetilde{F},\nstdN)\models \Phi_{\mathrm{H}}$, that is, the following proposition holds.
\begin{proposition}\label{pr:howson_nonstd}
    If two subgroups $G_1,G_2\le \widetilde{F}$ are nonstandardly generated, then so is their intersection $G_1\cap G_2$.
\end{proposition}
Notice that this is not an entirely satisfactory generalization of the Howson property of free groups, since if the subgroups $G_1,G_2$ are generated by finite tuples $\bar g^{(1)}, \bar g^{(2)}$, respectively (as opposed to infinite nonstandard tuples), then the above proposition guarantees only the generation of $G_1\cap G_2$ by a nonstandard, not necessarily finite, tuple.
Below we offer a way to refine both the formula $\Phi_{\mathrm{H}}$ and the subsequent proposition.

Recall that there is an easy to see inequality connecting rank (number of free generators) of $G_1$, $G_2$, and $G_1\cap G_2$: $\mathop{\mathrm{rk}} G_1\cap G_2-1\le 2(\mathop{\mathrm{rk}} G_1-1)(\mathop{\mathrm{rk}} G_2-1)$ (see, for example,~\cite{Neumann:1957}). Consider an extended version $\Phi_{\mathrm{H}}^+$ of $\Phi_{\mathrm{H}}$:
\[
\begin{split}
\Phi_{\mathrm{H}}^+=\forall \bar y^{(1)}, \bar y^{(2)}\
\exists \bar y\
[&
\ell(y)\le 2(\ell(y^{(1)}-1)(\ell(y^{(2)}-1)+1 \\
&\wedge \forall x\in F\ \varphi_{\mathrm{memb}}(x, \bar y^{(1)})\wedge \varphi_{\mathrm{memb}}(x, \bar y^{(2)})\leftrightarrow \varphi_{\mathrm{memb}}(x, \bar y)].
\end{split}
\]
Then $\widetilde{\mathbb S}(\widetilde{F},\nstdN)\models \Phi_{\mathrm{H}}^+$ can be recorded as the following proposition.
\begin{proposition}\label{pr:howson_finite1}
    If two subgroups $G_1,G_2\le \widetilde{F}$ are nonstandardly generated by $m_1$ and $m_2$ elements of $\widetilde{F}$, respectively, then their intersection $G_1\cap G_2$ is nonstandardly generated by at most $2(m_1-1)(m_2-1)+1$ elements of $\widetilde{F}$.
\end{proposition}
The above version does not have the deficiency we pointed out for Proposition~\ref{pr:howson_nonstd}. Separately, we note that the bound improves to at most $(m_1-1)(m_2-1)+1$ if we take into account the positive resolution of the Hanna Neumann conjecture~\cite{Mineyev:2012,Friedman:2015}.


\section{Nonstandard Groups}\label{se:presentations}
As we mention at the end of Section~\ref{se:free_group}, the construction therein applies to an arbitrary finitely generated structure with arithmetic multiplication table.
Consider an interpretation of such a group $G$ in $\mathbb S(\mathbb Z,\mathbb N)$, which allows us to obtain a nonstandard model $\widetilde{G}=G(\nstdN)$.
The latter is well-defined by Theorem~\ref{th:equiv2}.
Similar to Theorems~\ref{th:free_group_ultrapower} and~\ref{th:free_group_ultrapower_gen}, assuming the respective versions of the continuum-hypothesis, we can straightforwardly see that $G^*\cong \widetilde{G}$, where $G^*$ denotes the appropriate ultrapower of $G$. To this end, we formulate the following theorem.
\begin{theorem}\label{th:arbitrary_group_ultrapower_gen}
Let $G$ be a finitely generated group with arithmetic word problem. Suppose $\kappa$ is an infinite cardinal such that $2^\kappa=\kappa+$. Let $D$ be a non-principal ultrafilter on $\kappa$ and $G^*\cong \prod\limits_{i\in \omega}G\ /D$ and $\mathbb Z^*\cong \prod\limits_{i\in \omega}\mathbb Z\ /D$ be the corresponding ultrapowers. Then $G^*\cong \widetilde F$, where $\widetilde{G}$ denotes the nonstandard model $G(\mathbb Z^*)$.
\end{theorem}
\begin{proof}
The proof repeats that of Theorem~\ref{th:free_group_ultrapower} with $\kappa$ instead of $\omega$ and $G$ instead of $F$.
\end{proof}

We note that $G\le \widetilde{G}$ and that $G$ is an elementary submodel of $\widetilde{G}$. 

\begin{theorem}\label{th:elementary_submodel}
Let $G$ be a finitely generated group with arithmetic word problem. Then $G$ is an elementary submodel of $\widetilde{G}=G(\widetilde{\mathbb Z})$.
\end{theorem}
\begin{proof}
Follows since $\widetilde{G}$ is interpreted in $\nstdN$ and $\mathbb N$ is an elementary substructure of~$\nstdN$.
\end{proof}

\begin{theorem}\label{th:equivalent_models}
    If $G$ be a finitely generated group with arithmetic word problem and $G'\equiv G$, then there is a nonstandard group $\widetilde{G}$ such that $G'$ is an elementary submodel of $\widetilde{G}$ and $|\widetilde{G}|\le |G'|$.
\end{theorem}
\begin{proof} This is a particular case of~\cite[Theorem 5, Section 5.3]{Daniyarova-Myasnikov:I}.
\end{proof}

\begin{theorem}\label{th:saturated}
    Let $G$ be a finitely generated group with arithmetic word problem.  Suppose $\kappa$ is an infinite cardinal such that $2^\kappa=\kappa+$. If $G'\equiv G$, $G'$ is saturated, and $|G'|=\kappa$, then $G'\cong \widetilde{G}$ for some nonstandard model $\widetilde{G}$ of $G$.
\end{theorem}
\begin{proof}
    Consider an ultrapower $\widetilde{\mathbb Z}$ of $\mathbb Z$ of cardinality $\kappa$. Then $\widetilde{G}=G(\widetilde{\mathbb Z})$ is elementarily equivalent to $G$, is saturated (since $2^\kappa=\kappa+$), and has cardinality at most $\kappa$. Further, $|\widetilde{G}|\ge \kappa$, since $\widetilde{\mathbb S}(\widetilde{\mathbb Z}, \nstdN)$ satisfies formula $\forall n\ \exists \bar x\in \List_{\widetilde{G}}\ \forall i,j\le \ell(\bar x)\ x_i\neq x_i$.
    
    Therefore, since two elementarily equivalent saturated models are isomorphic, we have $G'\cong \widetilde{G}$.
\end{proof}

In fact, the above statements are not specific to groups. For example, the same reasoning as in Theorem~\ref{th:arbitrary_group_ultrapower_gen} allows us to obtain a similar result for arbitrary algebraic structures.
\begin{theorem}\label{th:arbitrary_structure_ultrapower_gen}
Let $\mathbb A$ be a finitely generated algebraic structure of finite signature with arithmetic multiplication table. Suppose $\kappa$ is an infinite cardinal such that $2^\kappa=\kappa+$. Let $D$ be a non-principal ultrafilter on $\kappa$ and $\mathbb A^*\cong \prod\limits_{i\in \omega}\mathbb A\ /D$ and $\mathbb Z^*\cong \prod\limits_{i\in \omega}\mathbb Z\ /D$ be the corresponding ultrapowers. Then $\mathbb A^*\cong \widetilde{\mathbb A}$, where $\widetilde{\mathbb A}$ denotes the nonstandard model $\mathbb A(\mathbb Z^*)$.
\end{theorem}
\begin{proof}
The proof repeats that of Theorem~\ref{th:free_group_ultrapower} with $\kappa$ instead of $\omega$ and $\mathbb A$ instead of~$F$.
\end{proof}

\subsection{Definable homomorphisms}\label{se:def_hom}
Suppose $\widetilde{G_1}$ and $\widetilde{G_2}$ are two nonstandard groups, with respective interpretation codes $\Gamma_1$ and $\Gamma_2$. If $\varphi:\widetilde{G_1}\to\widetilde{G_2}$ is a group homomorphism s.t. the set $\{(g,\varphi(g)\mid g\in\widetilde{G_1}\}$ is definable in $\widetilde{\mathbb S}(\widetilde{\mathbb Z},\nstdN)$, then we simply say that $\varphi$ is \emph{definable}. In the case a finitely generated $\widetilde{G_1}$, this property is equivalent to respecting nonstandard products.

\begin{proposition}\label{pr:nstd_hom}
Let $\nstdN\equiv\mathbb N$, and let $\widetilde{G_1}\cong \Gamma_1(\nstdN)$, $\widetilde{G_2}\cong \Gamma_2(\nstdN)$ be two nonstandard groups. If $\widetilde{G_1}$ is finitely generated, a group homomorphism $\widetilde{G_1}\to\widetilde{G_2}$ is definable if and only if for every nonstandard tuple $\bar g\in\mathbb \List(\widetilde{G_1},\nstdN)$, $\varphi(\prod \bar g)=\prod \varphi(\bar g)$, where $\varphi(\bar g)$ denotes element-wise application of $\varphi$.
\end{proposition}
\begin{proof}
Suppose that $\varphi$ is definable.
Then the following expression can be rewritten as a formula $\Phi(\bar g)$ of $\widetilde{\mathbb S}(\widetilde{\mathbb Z},\nstdN)$:
\[
\forall 1\le j\le \ell(\bar g)\ \varphi(\prod_{i=1}^j g_i) = \prod_{i=1}^j\varphi(g_i).
\]
By induction on $j$, $\widetilde{\mathbb S}(\widetilde{\mathbb Z},\nstdN)\models \Phi(\bar g)$, since $\varphi$ is a homomorphism.

Conversely, suppose $\varphi$ respects nonstandard products. Suppose $\widetilde{G_1}$ is generated by the finitely many elements $g_1,\ldots,g_k$, with $\varphi(g_1)=h_1,\ldots, \varphi(g_k)=h_k$. Then for $g\in\widetilde{G_1}, h\in\widetilde{G_2}$, we have $\varphi(g)=h$ if and only if there are two tuples $(i_1,\ldots,i_m)$ and $(\varepsilon_1,\ldots,\varepsilon_m)$, $\varepsilon_j=\pm 1$, in $\widetilde{\mathbb S}(\widetilde{\mathbb Z},\nstdN)$ s.t. $g=\prod (g_{i_1}^{\varepsilon_1},\ldots, g_{i_m}^{\varepsilon_m})$ and $h=\prod (h_{i_1}^{\varepsilon_1},\ldots, h_{i_m}^{\varepsilon_m})$. The latter can be expressed by a formula of $\widetilde{\mathbb S}(\widetilde{\mathbb Z},\nstdN)$.
\end{proof}

\subsection{Nonstandard presentations}
In the case when $G$ is finitely presented $G=\langle x_1,\ldots, x_k\mid R\rangle\cong F/\llangle R\rrangle$, we can define the group $\widetilde{G}$ in terms of the nonstandard free group on free generators $x_1,\ldots,x_k$, similarly to the standard case. Indeed, it is straightforward to define in $\mathbb S(F,\mathbb N)$ the normal closure $R^F$ of $R$ and the subgroup generated by $R^{{F}}$:
\[
\begin{split}
    R^F &= \{ w^{-1} r w: w\in F, r\in R\},\\
    \llangle R\rrangle&=\left\{\left.\prod (g_1,g_2,\ldots,g_n) \right| g_i^{\pm 1}\in R^{F}\right\},
\end{split}
\]
that is, $\llangle R\rrangle$ is the usual normal closure of $R$.
Now, we define a relation $\sim_R$ on $F$ as
\begin{equation}\label{eq:quotient}
    g_1\sim_R g_2\iff g_1=g_2h,\ h\in \llangle R\rrangle.
\end{equation}
Finally, the group operations of the quotient group $\langle x_1,\ldots,x_k\rangle/ \llangle R\rrangle$ can be straightforwardly defined in $\mathbb S(F,\mathbb N)$ (in fact, in $F$, modulo the weak second-order formulas to define $\llangle R^F\rrangle$).
This gives an interpretation of $G$ in $\mathbb S(F,\mathbb N)$, therefore in $\mathbb S(\mathbb Z,\mathbb N)$ by Proposition~\ref{pr:list_interpretation}, therefore in $\mathbb N$, say $G\cong \Gamma_1(\mathbb N)$.
The structure $\Gamma_1(\nstdN)$ is isomorphic to $\widetilde{G}$ by Theorem~\ref{th:equiv1}.
Notice that over $\nstdN$, the element $h$ in~\eqref{eq:quotient} becomes a nonstandard product of conjugates of relators.
In particular, in the nonstandard case, the subgroup $\llangle R\rrangle$ is the nonstandard subgroup generated by $R^{\widetilde{F}}$.
Therefore, the as a (standard) group, $\Gamma_1(\nstdN)$ is isomorphic to the (standard) group quotient $\widetilde{F}/\llangle R\rrangle$.
Moreover, by Proposition~\ref{pr:nstd_hom} it is straightforward to check that the mapping $x_i\to x_i\llangle R\rrangle$ extends uniquely to a definable homomorphism $\varphi:\widetilde{F}\to\widetilde{G}$ with the kernel $\ker\varphi=\llangle R\rrangle$, since the weak second order formulas that record that hold for $F$.
\begin{theorem}\label{th:first_iso} Let $F$ be the free group on a finite set $X=\{x_1,\ldots,x_k\}$.
Let $G=\langle X\mid R\rangle$ be a finitely generated recursively presented group.
Let $\nstdN\equiv \mathbb N$ and let $\widetilde{F}=F(\nstdN)$ and $\widetilde{G}=G(\nstdN)$ be the respective nonstandard groups.
Let $\llangle R\rrangle$ be the nonstandard subgroup of $\widetilde{F}$ generated by $R^{\widetilde{F}}$.
Then the mapping $x_i\to x_i\llangle R\rrangle$ extends uniquely to a definable homomorphism $\varphi$ from $\widetilde{F}$ onto $\widetilde{G}$ with kernel $\ker \varphi = \llangle R\rrangle$.
\end{theorem}
Notice that it follows that as (standard) groups, $\widetilde{G}\cong \widetilde{F}/\llangle R\rrangle$.
Therefore, we can think of $\widetilde{G}$ as the quotient group of $\widetilde{F}$ by the nonstandard normal closure of $R$. With this in mind, we denote $\widetilde{G}=\langle x_1,\ldots, x_k\mid R\rangle^\sim$.

A nonstandard analogue of the classic result on group words representing the same group element holds as stated below.
\begin{proposition}
Suppose $\bar x=(x_1,\ldots,x_k)$ is a tuple of elements of $\widetilde{G}$ s.t. $\widetilde{G}=\langle x_1,\ldots, x_k\mid R\rangle^\sim$.
Suppose $\bar a$ and $\bar b$ are two nonstandard tuples of elements of $\bar x$ and their inverses such that $\prod \bar a=\prod \bar b$ in $\widetilde{G}$. Then there is a nonstandard tuple of nonstandard tuples $(\bar a= \bar c^1,\ldots, \bar b = \bar c^r)$ s.t. $\bar c^i$ and $\bar c^{i+1}$ differ by a single free cancellation or insertion, or by a single deletion or insertion of a relator from $R$ or its inverse.
\end{proposition}
\begin{proof}
    Notice that the statement can be written as a formula $\varphi$ of $\widetilde{\mathbb S}(\widetilde{G}, \nstdN)$.
    Since $\mathbb S(G,\mathbb N)\models \varphi$, it follows $\widetilde{\mathbb S}(\widetilde{G}, \nstdN)\models \varphi$.
\end{proof}


Since groups with finite presentations have arithmetic word problem, Theorems~\ref{th:elementary_submodel}, \ref{th:equivalent_models}, \ref{th:saturated} apply to such groups, which we record as the following corollaries.


\begin{corollary}\label{co:elementary_submodel}
Suppose $\widetilde{G}=\langle x_1,\ldots, x_k\mid R\rangle^\sim$. Then $G=\langle x_1,\ldots, x_k\mid R\rangle$ is an elementary submodel of $G$.
\end{corollary}

\begin{corollary}\label{co:equivalent_models}
    If $G=\langle x_1,\ldots, x_k\mid R\rangle$ is a group and $G'\equiv G$, then there is a nonstandard group $\widetilde{G}=\langle x_1,\ldots, x_k\mid R\rangle^\sim$ such that $G'$ is an elementary submodel of $\widetilde{G}$ and $|\widetilde{G}|\le |G'|$.
\end{corollary}

\begin{corollary}\label{co:saturated}
    Let $G=\langle x_1,\ldots, x_k\mid R\rangle$. Suppose $\kappa$ is an infinite cardinal such that $2^\kappa=\kappa+$. If $G'\equiv G$, $G'$ is saturated, and $|G'|=\kappa$, then $G'\cong \widetilde{G}$ for some nonstandard model $\widetilde{G}$ of $G$.
\end{corollary}
\bibliographystyle{plain}
\bibliography{biblio}

\end{document}